\documentclass[11pt]{article}

\usepackage{amsmath,amsthm,verbatim,amssymb,amsfonts,amscd, graphicx}
\usepackage{graphics}
\usepackage{url}
\usepackage{authblk}
\usepackage{upgreek}
\usepackage{cite}
\theoremstyle{plain}
\newtheorem{theorem}{Theorem}[section]

\newtheorem{corollary}[theorem]{Corollary}
\newtheorem{definition}[theorem]{Definition}

\newtheorem{lemma}[theorem]{Lemma}
\newtheorem{proposition}[theorem]{Proposition}
\newtheorem{remark}[theorem]{Remark}
\newtheorem*{theorem*}{Theorem}
\numberwithin{equation}{section}
\DeclareMathOperator*{\Hess}{Hess}
\DeclareMathOperator*{\tr}{tr}
\DeclareMathOperator*{\divr}{div}

\begin{document}

\title{A regularity criterion for the Navier-Stokes equation involving only the middle eigenvalue of the strain tensor}
\author[1]{Evan Miller}
\affil[1]{The University of Toronto, Department of Mathematics, ep.miller@mail.utoronto.ca}
\maketitle

\begin{abstract}
This manuscript derives an evolution equation for the symmetric part of the gradient of the velocity (the strain tensor) in the incompressible Navier-Stokes equation on $\mathbb{R}^3$, and proves the existence of $L^2$ mild solutions to this equation.
We use this equation to obtain a simplified identity for the growth of enstrophy for mild solutions that depends only on the strain tensor, not on the nonlocal interaction of the strain tensor with the vorticity. The
resulting identity allows us to prove a new family of scale-critical, 
necessary and sufficient conditions for the blow-up of a solution at some finite time $T_{max}<+\infty$,  which depend only on the history of the positive part of the second eigenvalue of the strain matrix.  Since this matrix is trace-free,
this severely restricts the 
geometry of any finite-time blow-up. This regularity criterion provides analytic evidence of the numerically observed tendency of the vorticity to align with the eigenvector corresponding to the middle eigenvalue of the strain matrix. This regularity criterion also allows us to prove as a corollary a new scale critical, one component type, regularity criterion for a range of exponents for which there were previously no known critical, one component type regularity criteria. Furthermore, our analysis permits us to extend the known time of existence of smooth solutions with fixed initial enstrophy $E_0=\frac{1}{2}\left\|\nabla \otimes u^0\right \|_{L^2}^2$
by a factor of $4{,}920{.}75$---although the previous constant in the literature was not expected to be close to optimal, so this improvement is less drastic than it sounds, especially compared with numerical results.  Finally, we will prove the existence and stability of blow-up for a toy model ODE for the strain equation.
\end{abstract}

\section*{Acknowledgments}
This work was supported by the Ontario Trillium Scholarship, the Coxeter Graduate Scholarship, the University of Toronto Department of Mathematics, and the CUPE local 3902 and CUPE national strike funds. I would like to thank my adviser, Prof. Robert McCann, for his support and advice throughout this work. I would also like to thank the anonymous referee for their extremely thorough and helpful report.

This is a preprint of a paper published in the Archive for Rational Mechanics and Analysis. The final published version is available at \url{https://doi.org/10.1007/s00205-019-01419-z.}
\pagebreak

\section{Introduction}

The Navier-Stokes equations are the fundamental equations of fluid dynamics. The incompressible Navier-Stokes equations are given by
\begin{equation} \label{Navier}
\begin{split}
\partial_t u- \Delta u + (u \cdot \nabla)u + \nabla p&=0,  \\
\nabla \cdot u&=0.
\end{split}
\end{equation}
where $u \in \mathbb{R}^3$ denotes the velocity and $p$ the pressure. For incompressible flow with no external force, the pressure is completely determined by $u$ by taking the divergence of both sides of the equation, which yields
\begin{equation}
- \Delta p =\sum_{i,j=1}^3 \frac{\partial u_i}{\partial x_j}
\frac{\partial u_j}{\partial x_i}.
\end{equation}
The Navier-Stokes equation---which is best viewed as an evolution equation on the space of divergence free vector fields rather than as a system of equations---is nonlocal because of the nonlocal dependence of $\nabla p$ on $u.$
Much about the solutions of the Navier-Stokes equation is unknown, including uniqueness and regularity.
The main barrier is the fact that the energy equality, which states that for any smooth solution $u$ of the Navier-Stokes equation
\begin{equation}
    \frac{1}{2}\|u(\cdot,t)\|_{L^2}^2+
    \int_0^t\|\nabla \otimes u(\cdot,\tau)\|_{L^2}^2d\tau
    =\frac{1}{2}\left\|u^0\right\|_{L^2}^2,
\end{equation}
implies $L^\infty_t L^2_x\cap L^2_t \dot{H}^1_x$ control on solutions, but this control is supercritical in three spatial dimensions with respect to the scaling $u^\lambda(x,t)=\lambda u(\lambda x,\lambda^2 t)$ that preserves the solution set of \eqref{Navier}.

In this paper, we will prove new conditional regularity results based on an analysis of the role of the strain tensor.
We will begin by defining a few matrices.
The gradient tensor will be given by 
\begin{equation}(\nabla \otimes u)_{ij}=\frac{\partial u_j}{\partial x_i}.
\end{equation}
The symmetric part $S,$ which we will refer to as the strain tensor, will be given by
\begin{equation}
S_{ij}=\frac{1}{2}\left(\frac{\partial u_j}{\partial x_i}+\frac{\partial u_i}{\partial x_j}\right),
\end{equation}
and the anti-symmetric part $A$ will be given by
\begin{equation}A_{ij}=\frac{1}{2}\left(\frac{\partial u_j}{\partial x_i}-\frac{\partial u_i}{\partial x_j}\right).
\end{equation}
It is immediately clear that $S$ is symmetric, $A$ is anti-symmetric, and that $\nabla \otimes u=S+A.$  We also define the differential operator $\nabla_{sym}$ as the map from a vector to the symmetric part of its gradient tensor:  $\nabla_{sym}(v)_{ij}= \frac{1}{2}\left(\frac{\partial v_j}{\partial x_i}+\frac{\partial v_i}{\partial x_j}\right)$. We will also note that in three spatial dimensions the anti-symmetric matrix $A$ can be represented as a vector, which we will call the vorticity. The vorticity, given by the curl of the flow $\omega=\nabla \times u$, plays a very important role in fluid mechanics. It is related to $A$ by
\begin{equation} \label{AntiVort}
A=\frac{1}{2} \left (
\begin{matrix}
0 & \omega_3 & -\omega_2 \\
-\omega_3 & 0 & \omega_1 \\
\omega_2 & -\omega_1 & 0 \\
\end{matrix}
\right ).
\end{equation}

The equation for the evolution of vorticity $\omega$ is written as follows:
\begin{equation} \label{NavierVorticity}
\partial_t \omega- \Delta \omega+ (u\cdot \nabla) \omega- S \omega=0.
\end{equation}
The vortex stretching term $S \omega$ is often written $(\omega \cdot \nabla) u,$ but it is clear from \eqref{AntiVort} that $A \omega=0,$ therefore
\begin{align}
(\omega \cdot \nabla) u&=\left(S+A \right)\omega\\
&=S\omega.
\end{align}
The vorticity equation \eqref{NavierVorticity} is also a nonlocal equation because $u=\nabla \times (- \Delta)^{-1} \omega$. The vortex stretching term $S \omega$ is the quadratic nonlinearity, but does not have a very natural structure relative to the space of divergence free vector fields.  For this reason, we will consider the evolution equation for the strain tensor. In three spatial dimensions the evolution equation for the strain is given by
\begin{equation} \label{NavierStrain}
\partial_t S+ (u \cdot \nabla)S  -\Delta S+S^2+\frac{1}{4}\omega \otimes \omega - \frac{1}{4}|\omega|^2 I_3+ \Hess(p)=0,
\end{equation}
as we will show in the next section.

The strain tensor has already been studied \cite{VorticityAlignmentResults,ConcentratedVorticity,NonlocalStrain} in terms of its importance for enstrophy production in the vorticity equation \eqref{NavierVorticity}.  The evolution equation for the strain tensor, while it requires additional terms, has a quadratic nonlinearity whose structure is far better from an algebraic point of view.  This is because while a vector cannot be squared, and the square of an anti-symmetric matrix (the other representation of vorticity) is a symmetric matrix, the square of a symmetric matrix is again a symmetric matrix.

It is well known that 
\begin{equation} \label{vrt}
\partial_t \frac{1}{2} \|\omega(\cdot,t)\|_{L^2}^2=-\|\omega\|_{\dot{H^1}}^2+
\left < S, \omega \otimes \omega \right >.
\end{equation}
Using the Sobolev embedding of $\dot{H}^1\left(\mathbb{R}^3 \right )$ into 
$L^6\left(\mathbb{R}^3 \right )$ it follows from \eqref{vrt} that
\begin{equation}
\partial_t \|\omega(\cdot,t)\|_{L^2}^2 \leq C \|\omega(\cdot,t)\|_{L^2}^6,
\end{equation}
which is sufficient to guarantee regularity at least locally in time, but cannot prevent blowup because it is a cubic differential inequality.
Enstrophy can be defined equivalently as
\begin{align}
E(t)&=\frac{1}{2}\|\nabla \otimes u(\cdot,t)\|_{L^2}^2\\
&=\frac{1}{2}\|\omega(\cdot,t)\|_{L^2}^2\\
&=\|S(\cdot,t)\|_{L^2}^2.
\end{align}
We will prove the equivalence of these definitions in section 3.
In this paper, we will prove an identity for enstrophy growth:
\begin{equation} \label{st}
\partial_t \|S(\cdot,t)\|_{L^2}^2=-2\|S\|_{\dot{H^1}}^2-
\frac{4}{3}\int \tr(S^3).
\end{equation}
Using the fact that $S$ must be trace free, because $\tr(S)=\nabla \cdot u=0,$ this identity can also be expressed in terms of the determinant of $S$ as
\begin{equation} \label{stDet}
\partial_t \|S(\cdot,t)\|_{L^2}^2=-2\|S\|_{\dot{H^1}}^2-
4\int \det(S).
\end{equation}

The nonlinearity in \eqref{st} is still of the same degree as in \eqref{vrt}. Both nonlinearities are of degree $3,$ and so cannot be controlled by the dissipation in either case, however the identity \eqref{st} does have several advantages. First, unlike \eqref{vrt}, this identity is entirely local. The identity \eqref{vrt} is nonlocal with a singular integral kernel, because $S$ can be determined in terms of $\omega$ with a zeroth order pseudo-differential operator, 
$S=\nabla_{sym} (- \Delta)^{-1}\nabla \times \omega.$
The identity \eqref{st} also reveals very significant information about the eigenvalues of the strain tensor $S$ and their relation to blowup.
Note that throughout this paper $\dot{H}^s$ will refer to the homogeneous Sobolev space with the norm
\begin{equation}
\|f\|_{\dot{H}^s}^2= \left < f, (- \Delta)^s f \right >,
\end{equation}
and we will take the magnitude of a matrix, $M\in \mathbb{R}^{3 \times 3},$ 
to be the Euclidean norm
\begin{equation}
|M|^2=\sum_{i,j=1}^3M_{ij}^2.
\end{equation}

Finally, the identity \eqref{stDet} will lead to a new regularity criterion that encodes information about the geometric structure of potential blow-up solutions. The main result of this paper is:
\begin{theorem}[Middle eigenvalue of strain characterizes blowup time] \label{IntroEigenRegCrit}
Let $u \in C\left ([0,T];\dot{H}^1\left(\mathbb{R}^3\right ) \right )$
for all $T<T_{max}$ be a mild solution to the Navier-Stokes equation, and let $\lambda_1(x) \leq \lambda_2(x) \leq \lambda_3(x)$ be the eigenvalues of the strain tensor $S(x)=\nabla_{sym}u(x).$ Let $\lambda_2^+(x)=\max \{\lambda_2(x),0\}.$ If $\frac{2}{p}+\frac{3}{q}=2,$ with $\frac{3}{2}<q \leq +\infty,$ then
\begin{equation}
\|u(\cdot,T)\|_{\dot{H}^1}^2 \leq \left\|u^0\right\|_{\dot{H}^1}^2
\exp \left (C_q \int_{0}^T\|\lambda_2^+(\cdot,t)\|_{L^q\left (\mathbb{R}^3\right )}^p dt \right ),
\end{equation}
with the constant $C_q$ depending only on $q.$
In particular if $T_{max}<+\infty,$ where $T_{max}$ is the maximal existence time for a smooth solution, then
\begin{equation}
\int_{0}^{T_{max}}\|\lambda_2^+(\cdot,t)\|_{L^q(\mathbb{R}^3)}^p dt= + \infty.
\end{equation}
\end{theorem}

It goes back to the classic work of Kato \cite{KatoL3} that smooth solutions must exist locally in time for any initial data $u^0 \in L^q\left (\mathbb{R}^3 \right ),$ $\nabla \cdot u^0=0,$ when $q>3$. In particular, this implies that a smooth solution of the Navier-Stokes equations developing singularities in finite time requires that the $L^q$ norm of $u$ must blow up for all $q>3.$ This was extended to the case $q=3$ by Escauriaza, Seregin, and \v{S}ver\'ak \cite{SereginL3}. The regularity criteria implied by the local existence of smooth solutions for initial data in $L^q \left (\mathbb{R}^3\right)$ when $q>3$ are all subcritical with respect to the scaling that preserves the solution set of the Navier-Stokes equations:
\begin{equation}
u^\lambda(x,t)=
\lambda u(\lambda x, \lambda^2t).
\end{equation}
If $u$ is a solution to the Navier-Stokes equations on $\mathbb{R}^3$, then so is $u^\lambda$ for all $\lambda>0,$ although the time interval may have to be adjusted, depending on what notion of a solution (Leray-Hopf \cite{Leray}, mild, strong \cite{Kato}) we are using.
$L^3 \left ( \mathbb{R}^3 \right)$ is the scale critical Lebesgue space for the Navier-Stokes equations, so the Escauriaza-Seregin-\v{S}ver\'ak condition is scale critical.

Critical regularity criteria for solutions to the Navier-Stokes equations go back to the work of Prodi, Serrin, and Ladyzhenskaya \cite{Serrin,Ladyzhenskaya,Prodi}, who proved that if a smooth solution $u$ of the Navier-Stokes equation blows up in finite time, then 
\begin{equation} \label{Serrin}
\int_0^{T_{max}}\|u\|_{L^q}^p dt=+ \infty,
\end{equation}
where $\frac{2}{p}+\frac{3}{q}=1,$ and $3<q \leq + \infty$. This result was then extended in the aforementioned Escauriaza-Seregin-\v{S}ver\'ak paper\cite{SereginL3} to the endpoint case $p=+ \infty, q=3.$ They proved that if a smooth solution $u$ of the Navier-Stokes equation blows up in finite time, then 
\begin{equation}
\limsup_{t \to T_{max}}\|u(\cdot,t)\|_{L^3(\mathbb{R}^3)}=+\infty.
\end{equation}
Gallagher, Koch, and Planchon \cite{ProfileDecomp} also proved the above statement  using a different approach based on profile decomposition.
The other endpoint case of this family of regularity criteria is the Beale-Kato-Majda criterion \cite{BealeKatoMajda}, which holds for Euler as well as for Navier-Stokes, and states that if a smooth solution to either the Euler or Navier-Stokes equations develops singularities in finite time, then 
\begin{equation}
\int_0^{T_{max}}\|\omega(\cdot, t)\|_{L^\infty} dt= + \infty.
\end{equation} 
This result was also extended to the strain tensor \cite{KatoStrain}.

The regularity criterion in Theorem \ref{IntroEigenRegCrit} also offers analytical evidence of the numerically observed tendency \cite{VorticityAlignmentResults} of the vorticity to align with the eigenvector of $S$ corresponding to the intermediate eigenvalue $\lambda_2.$ If it is true that the vorticity tends to align with the intermediate eigenvalue, we would heuristically expect that
\begin{equation}
tr\big (S(x) \omega(x) \otimes \omega(x)\big ) \sim  \lambda_2(x) |\omega(x)|^2.
\end{equation}
We would then heuristically expect that
\begin{equation}
\left <S; \omega \otimes \omega \right > \sim 
\int_{\mathbb{R}^3} \lambda_2(x)|\omega(x)|^2 dx,
\end{equation}
and so we would expect that there would be some inequality of the form
\begin{equation}
\left <S; \omega \otimes \omega \right > \leq C 
\int_{\mathbb{R}^3}\lambda_2^+(x)|\omega(x)|^2 dx.
\end{equation}
This is all, of course, entirely heuristic, but it is interesting that the regularity criterion we have proven is precisely of the form that would be predicted by the observed tendency of the vorticity to align with the eigenvector associated with the intermediate eigenvalue. This suggests that significant information about the geometric structure of incompressible flow is encoded in the regularity criterion in Theorem \ref{IntroEigenRegCrit}.

The family of regularity criteria in \eqref{Serrin} has since been generalized to the critical Besov spaces \cite{Gallagher,232,233,90,Phuc,Dallas}. These criteria have also been generalized to criteria controlling the pressure \cite{StruwePressure,XinweiPressure,SereginPressure}. In addition to strengthening regularity criteria to larger spaces, there have also been results not involving all the components of $u,$ for instance regularity criteria on the gradient of one component $\nabla u_j$ \cite{ZhouOne}, involving only the derivative in one direction, $\partial_{x_i}u$ \cite{Kukavica}, involving only one component $u_j$ \cite{Chemin,Chemin2}, involving only one component of the gradient tensor $\frac{\partial u_j}{\partial x_i}$ \cite{CaoTiti}, and involving only two components of the vorticity \cite{Vort2}. For a more thorough overview of the literature on regularity criteria for solutions to the Navier-Stokes equation see Chapter 11 in \cite{NS21}. \nocite{SereginBlowUp} 
We will discuss the relationship between these results and our main theorem further in section 5, where we will prove the following critical one direction type regularity criterion for a range of exponents for which no critical one component regularity criteria were previously known. First we must define, for any unit vector $v \in \mathbb{R}^3,$ $|v|=1,$ the directional derivative in the $v$ direction, which is given by
$\partial_v=v \cdot \nabla,$ and the $v$-th component of $u,$ which is given by $u_v=u\cdot v.$

\begin{theorem}[One direction regularity criterion] \label{OneDir}
Let $\left \{v_n(t) \right \}_{n\in \mathbb{N}}\subset \mathbb{R}^3$ with 
$|v_n(t)|=1.$ Let $\left \{\Omega_n(t) \right \}_{n\in \mathbb{N}}\subset \mathbb{R}^3$ be Lebesgue measurable sets such that for all $m \neq n,$
$\Omega_m(t) \cap \Omega_n(t)= \emptyset,$ and 
$\mathbb{R}^3= \bigcup_{n\in \mathbb{N}} \Omega_n(t).$
Let $u \in C\left([0,T];\dot{H}^1\left (\mathbb{R}^3\right ) \right ),$ for all $T<T_{max}$ be a mild solution to the Navier-Stokes equation.
If $\frac{2}{p}+\frac{3}{q}=2,$ with $\frac{3}{2}<q \leq + \infty,$ then
\begin{equation}
\|u(\cdot,T)\|_{\dot{H}^1}^2 \leq 
\left\|u^0\right\|_{\dot{H}^1}^2
\exp \left (C_q \int_{0}^T\left (\sum_{n=1}^\infty\left|\left|
\frac{1}{2}\partial_{v_n}u(\cdot,t)+ \frac{1}{2}\nabla u_{v_n}(\cdot,t)
\right|\right|_{L^q(\Omega_n(t))}^q \right )^\frac{p}{q} dt \right ),
\end{equation}
with the constant $C_q$ depending only on $q.$
In particular if the maximal existence time for a smooth solution $T_{max}<+\infty,$ then 
\begin{equation} \label{Locally2D}
\int_{0}^{T_{max}}\left (\sum_{n=1}^\infty \bigl \|
\partial_{v_n} u(\cdot,t)+ \nabla u_{v_n}(\cdot,t) \bigr 
\|_{L^q(\Omega_n(t))}^q \right )^\frac{p}{q} dt= + \infty.
\end{equation}
Note that if we take $v_n(t)= \left ( \begin{matrix}
0\\0\\1 \end{matrix} \right )$ for each $n \in \mathbb{N},$
then \eqref{Locally2D} reduces to
\begin{equation}
\int_{0}^{T_{max}}\|\partial_3 u(\cdot,t)+\nabla u_3(\cdot,t)
\|_{L^q(\mathbb{R}^3)}^p dt= + \infty.
\end{equation}
\end{theorem}

Theorem \ref{OneDir} is in fact a corollary of the following more general theorem, which states that for a solution of the Navier-Stokes equation to blow up, the strain must blow up in every direction.

\begin{theorem}[Blowup requires the strain to blow up in every direction] \label{IntroStrain}
Let $u \in C\left([0,T];\dot{H}^1\left (\mathbb{R}^3\right ) \right ),$ for all $T<T_{max}$ be a mild solution to the Navier-Stokes equation and let 
$v \in L^\infty\left (\mathbb{R}^3 \times [0,T_{max}];\mathbb{R}^3\right ),$ 
with $|v(x,t)|=1$ almost everywhere.
If $\frac{2}{p}+\frac{3}{q}=2,$ with $\frac{3}{2}<q \leq + \infty,$ then
\begin{equation}
\|u(\cdot,T)\|_{\dot{H}^1}^2 \leq 
\left\|u^0\right\|_{\dot{H}^1}^2
\exp \left (C_q \int_{0}^T\|S(\cdot,t)v(\cdot,t)\|_{L^q\left (\mathbb{R}^3\right )}^p dt \right ),
\end{equation}
with the constant $C_q$ depending only on $q.$
In particular if the maximal existence time for a smooth solution $T_{max}<+\infty,$ then 
\begin{equation}
\int_{0}^{T_{max}}\|S(\cdot,t)v(\cdot,t)\|_{L^q(\mathbb{R}^3)}^p dt= + \infty.
\end{equation}
\end{theorem}

Note that like the Prodi-Serrin-Ladyzhenskaya regularity criterion, the regularity criteria we prove on $\lambda_2^+$ and $\partial_3 u+ \nabla u_3$ are critical with respect to scaling. The reason we require that $\frac{2}{p}+\frac{3}{q}=2,$ not $\frac{2}{p}+\frac{3}{q}=1$ is because $\lambda_2$ is an eigenvalue of $S,$ and therefore scales like $\nabla \otimes u,$ not like $u.$ In addition, both regularity criteria can be generalized to the Navier-Stokes equation with an external force $f\in L^2_t L^2_x,$ which will be discussed in section 5, but is left out of the introduction for the sake of brevity.

In section 2, we will derive the evolution equation for the strain tensor. In section 3, we will consider the relationship between the strain and the vorticity, and will prove a Hilbert space isometry between them. In section 4, we will use the evolution equation for the strain tensor and the Hilbert space isometry to prove the simplified growth identity for enstrophy. In section 5, we will prove Theorem \ref{IntroEigenRegCrit}, the regularity criterion in terms of $\lambda_2^+$ and obtain new one direction regularity criteria as immediate corollaries. In section 6, we will prove the existence and stability of blowup for a toy model ODE of the strain evolution equation, and examine the asymptotics of the spectrum of eigenvalues approaching blowup. 

Finally, we will note that while we have proven our results on the whole space, they apply equally on the torus, with more or less identical proofs. The only difference will be that some of the constants may have different values, as the sharp Sobolev constants may be different on the torus.

\begin{remark}
Following submission of this paper, the author learned of previous work by Dongho Chae on the role of the eigenvalues of the strain matrix in enstrophy growth in the context of the Euler equation \cite{ChaeStrain}.
In this paper, Chae proves that sufficiently smooth solutions to the Euler equation satisfy the following growth identity for enstrophy:
\begin{equation}\partial_t\|S(\cdot,t)\|_{L^2}^2=-4\int \det(S).\end{equation}
This is analogous to what we have proven for the growth of enstrophy for solutions of the Navier-Stokes equation \eqref{stDet} without the dissipation term, because the Euler equation has no viscosity. The methods used are somewhat different than ours; in particular the constraint space for the strain tensor and the evolution equation for the strain tensor are not used in \cite{ChaeStrain}. While it is possible to establish the identity \eqref{stDet} without an analysis of the constraint space, we expect the results characterizing the constraint space in this paper, particularly Proposition \ref{StrainSpace} and Proposition \ref{isometry}, to be useful in future investigations. Chae also proves the $q=+\infty$ case of the regularity criterion in Theorem \ref{IntroEigenRegCrit}, but this criterion is new for the rest of the range of parameters. We will discuss the relationship between our method of proof and that in \cite{ChaeStrain} in more detail after we have proven the identity \eqref{stDet}, which is Corollary \ref{DET2} in this paper. The author would like to thank Alexander Kiselev for bringing Chae's paper to his attention.

\end{remark}

\section{Evolution equation for the strain tensor}
We will begin this section by deriving the Navier-Stokes strain equation \eqref{NavierStrain} in three spatial dimensions.
\begin{proposition}[Strain reformulation of the dynamics]
Suppose $u$ is a classical solution to the Navier-Stokes equation.  Then S$=\nabla_{sym}(u)$ is a classical solution to the Navier-Stokes strain equation
\begin{equation}
\partial_t S+ (u \cdot \nabla)S  -\Delta S+S^2+
\frac{1}{4}\omega \otimes \omega- \frac{1}{4}|\omega|^2I_3+\Hess(p)=0.
\end{equation}
\end{proposition}
\begin{proof}
We begin by applying the operator $\nabla_{sym}$ to the Navier-Stokes equation \eqref{Navier}; we find immediately that
\begin{equation}
\partial_t S -\Delta S +\Hess(p)+\nabla_{sym}\left((u \cdot \nabla)u \right)=0.
\end{equation}
It remains to compute $\nabla_{sym}\left( (u \cdot \nabla)u \right).$
\begin{equation}
\nabla_{sym}\left( (u \cdot \nabla)u \right)_{ij}= \frac{1}{2} \partial_{x_i} \sum_{k=1}^3 u_k \frac{\partial u_j}{\partial x_k} +\frac{1}{2} \partial_{x_j} \sum_{k=1}^3 u_k \frac{\partial u_i}{\partial x_k}. 
\end{equation}
\begin{equation}
\nabla_{sym}\left( (u \cdot \nabla)u \right)_{ij}= \sum_{k=1}^3 u_k \partial_{x_k} \left(\frac{1}{2}\left(\frac{\partial u_j}{\partial x_i}+\frac{\partial u_i}{\partial x_j}\right)  \right) + \frac{1}{2}\sum_{k=1}^3 
\frac{\partial u_k}{\partial x_i}\frac{\partial u_j}{\partial x_k}+ \frac{\partial u_i}{\partial x_k} \frac{\partial u_k}{\partial x_j}.
\end{equation}
We can see from our definitions of $S$ and $A$ that
\begin{align}
\left(S^2\right)_{ij}&=\frac{1}{4} \sum_{k=1}^3 \left(\frac{\partial u_k}{\partial x_i}+\frac{\partial u_i}{\partial x_k}\right)
\left(\frac{\partial u_j}{\partial x_k}+\frac{\partial u_k}{\partial x_j}\right)\\
&=
\frac{1}{4}\sum_{k=1}^3 
\frac{\partial u_k}{\partial x_i}\frac{\partial u_j}{\partial x_k}+ \frac{\partial u_i}{\partial x_k} \frac{\partial u_k}{\partial x_j}
+\frac{\partial u_k}{\partial x_i}\frac{\partial u_k}{\partial x_j}+ \frac{\partial u_i}{\partial x_k} \frac{\partial u_j}{\partial x_k},
\end{align}
and
\begin{align}
\left(A^2\right)_{ij}&= \frac{1}{4} \sum_{k=1}^3 \left(\frac{\partial u_k}{\partial x_i}-\frac{\partial u_i}{\partial x_k}\right)
\left(\frac{\partial u_j}{\partial x_k}-\frac{\partial u_k}{\partial x_j}\right)\\
&=
\frac{1}{4}\sum_{k=1}^3 
\frac{\partial u_k}{\partial x_i}\frac{\partial u_j}{\partial x_k}+ \frac{\partial u_i}{\partial x_k} \frac{\partial u_k}{\partial x_j}
-\frac{\partial u_k}{\partial x_i}\frac{\partial u_k}{\partial x_j}- \frac{\partial u_i}{\partial x_k} \frac{\partial u_j}{\partial x_k}.
\end{align}
Taking the sum of these two equation, we find that
\begin{equation}
\left ( S^2+A^2\right )_{ij}=  \frac{1}{2}\sum_{k=1}^3 
\frac{\partial u_k}{\partial x_i}\frac{\partial u_j}{\partial x_k}+ \frac{\partial u_i}{\partial x_k} \frac{\partial u_k}{\partial x_j}.
\end{equation}
From this we can conclude that
\begin{equation}
\nabla_{sym}\left( (u \cdot \nabla)u \right)= (u \cdot \nabla)S+  S^2+A^2.
\end{equation}
Recall that
\begin{equation}
A=\frac{1}{2} \left (
\begin{matrix}
0 & \omega_3 & -\omega_2 \\
-\omega_3 & 0 & \omega_1 \\
\omega_2 & -\omega_1 & 0 \\
\end{matrix}
\right ),
\end{equation}
 so we can express $A^2$ as
\begin{equation}
A^2= \frac{1}{4}\omega \otimes \omega-\frac{1}{4}|\omega|^2 I_3.
\end{equation}
This concludes the proof.
\end{proof}

We also can see that $\tr(S)=\nabla \cdot u=0$, so in order to maintain the divergence free structure of the flow, we require that the strain tensor be trace free.  For the vorticity the only consistency condition is that the vorticity be divergence free.  Any divergence free vorticity can be inverted back to a unique velocity field, assuming suitable decay at infinity, with $u=\nabla \times (- \Delta)^{-1} w$.  This is not true of the strain tensor, for which an additional consistency condition is required.

If we know the strain tensor $S$, this is enough for us to reconstruct the flow.  We take
\begin{align} \label{ConsistencyStep}
-2 \divr(S)&= -\Delta u -\nabla (\nabla \cdot u)\\
&= -\Delta u.
\end{align}
Therefore we find that
\begin{equation} \label{StreamReconstruction}
u=-2 \divr (-\Delta)^{-1}S.
\end{equation}
This allows us to reconstruct the flow $u$ from the strain tensor $S$, but it doesn't guarantee that if we start with a general trace free symmetric matrix, the $u$ we reconstruct will actually have this symmetric matrix as its strain tensor.  We will need to define a consistency condition guaranteeing that the strain tensor is actually the symmetric part of the gradient of some divergence free vector field.  This condition for the strain equation will play the same role that the divergence free condition plays in the vorticity equation. We will now define the subspace of strain matrices 
$L^2_{st} \subset L^2(\mathbb{R}^3;S^{3 \times 3}).$
\begin{definition}[Strain subspace]
We will define the subspace of strain matrices to be\begin{equation}
L^2_{st}=\left \{ \frac{1}{2}\nabla \otimes u+\frac{1}{2}(\nabla \otimes u)^*:
u\in \dot{H}^1 \left (\mathbb{R}^3; \mathbb{R}^3 \right), \nabla \cdot u=0 \right \}.
\end{equation}
\end{definition}

This subspace of $L^2(\mathbb{R}^3;S^{3 \times 3})$ can in fact be characterized by a partial differential equation, although in this case, it is significantly more complicated than the equation $\nabla \cdot u=0,$ that characterizes the space of divergence free vector fields.
\begin{proposition}[Characterization of the strain subspace] \label{StrainSpace}
Suppose $S \in L^2(\mathbb{R}^3;S^{3 \times 3}).$ Then $S \in L^2_{st}$ if and only if
\begin{align} \label{Consistency Condition}
-\Delta S +2 \nabla_{sym}\left (\divr(S) \right )&=0,\\
\tr(S)&=0.
\end{align}

Because by hypothesis we only have $S \in L^2,$ we will consider $S$ to be a solution to \eqref{Consistency Condition} if the condition is satisfied pointwise almost everywhere in Fourier space, that is if
\begin{equation}
|\xi|^2 \hat{S}(\xi)-(\xi \otimes \xi) \hat{S}(\xi)- \hat{S}(\xi)(\xi \otimes \xi)=0,
\end{equation}
almost everywhere $\xi \in \mathbb{R}^3.$
Also, note that the matrix PDE \eqref{Consistency Condition} can be written out in components as
\begin{equation} \label{sym3}
-\Delta S_{ij}+\sum_{k=1}^3 \partial_{x_i} \partial_{x_k} S_{kj}+ 
\partial_{x_j} \partial_{x_k} S_{ki}=0.
\end{equation}
\end{proposition}
\begin{proof}
First suppose $S \in L^2_{st},$ so there exists a $u \in \dot{H}^1,$ $\nabla \cdot u=0,$ such that
\begin{equation} \label{sym}
S=\nabla_{sym}u.
\end{equation}
As we have already shown, $\tr(S)=\nabla \cdot u=0.$
Next we will take the divergence of \eqref{sym}, and find that,
\begin{align}
-2 \divr(S)&= -2 \divr (\nabla_{sym}u)\\
&=-\Delta u- \nabla (\nabla \cdot u)\\
&=-\Delta u. \label{sym2}
\end{align}
Applying $\nabla_{sym}$ to \eqref{sym2} we find that
\begin{align}
-2 \nabla_{sym}(\divr(S))&=\nabla_{sym}(-\Delta u)\\
&= -\Delta S,
\end{align}
so the condition \eqref{Consistency Condition} is also satisfied.

Now suppose $tr(S)=0$ and $-\Delta S+ 2 \nabla_{sym}(\divr(S))=0.$ Define $u$ by
\begin{equation}
u= (-\Delta)^{-1}(-2 \divr(S)).
\end{equation}
Applying $\nabla_{sym}$ to this definition we find that
\begin{align}
\nabla_{sym}u&=(-\Delta)^{-1}\left( -2 \nabla_{sym}(\divr(S))\right )\\
&=(-\Delta)^{-1}(-\Delta S)\\
&=S.
\end{align}
Clearly $u \in \dot{H}^1$ because $S \in L^2$ and $(-\Delta)^{-1}(-2 \divr)$ is a pseudo-differential operator with order $-1.$
It only remains to show that $\nabla \cdot u=0.$
Next we will take the trace of \eqref{sym3} and find that
\begin{equation}
(\divr)^2(S)= \sum_{i,j=1}^3 \partial_{x_i} \partial_{x_j}S_{ij}=0.
\end{equation}
Using this we compute that
\begin{equation}
\nabla \cdot u= (-\Delta)^{-1} (-2 (\divr)^2(S))=0.
\end{equation}
This completes the proof.
\end{proof}
Note that the the consistency condition \eqref{Consistency Condition} is linear, so the set of matrices satisfying it form a subspace of $L^2$. The Navier-Stokes equation \eqref{Navier} and the vorticity equation \eqref{NavierVorticity} can best be viewed not as systems of equations, but as evolution equations on the space of divergence free vector fields. Similarly, we can view the Navier-Stokes strain equation \eqref{NavierStrain} as an evolution equation on $L^2_{st}.$

The Navier-Stokes strain equation has already been examined in \cite{Constantin,VorticityAlignmentResults}, however the consistency condition \eqref{Consistency Condition} does not play a role in this analysis.
In \cite{VorticityAlignmentResults}, the authors focus on the relationship between vorticity and the strain tensor in enstrophy production, as the strain tensor and vorticity are related by a linear zero order pseudo-differential operator,
$S=\nabla_{sym} (- \Delta)^{-1}\nabla \times \omega.$
However, the consistency condition is actually very useful in dealing with the evolution of the strain tensor, because a number of the terms in the evolution equation \eqref{NavierStrain} are actually in the orthogonal compliment of $L^2_{st}$ with respect to the $L^2$ inner product. This will allow us to prove an identity for enstrophy growth involving only the strain, where previous identities involved the interaction of the strain and the vorticity. We will now make an observation about what matrices in $L^2(\mathbb{R}^3;S^{3 \times 3})$ are in the orthogonal complement of $L^2_{st}$ with respect to the $L^2$ inner product.  
\begin{proposition}[Orthogonal subspaces] \label{isometry}
For all $f \in \dot{H}^2(\mathbb{R}^3),$ for all $g \in L^2(\mathbb{R}^3)$, and for all $S \in L^2_{st}$
\begin{equation}
\left < S, g I_3\right >=0,
\end{equation}

\begin{equation}
\left < S, \Hess(f)\right >=0.
\end{equation}
\end{proposition}
\begin{proof}
First we'll consider the case of $gI_3$.  Fix $S \in L^2_{st}$ and we'll take the inner product
\begin{equation}
\left <gI_3,S \right >=\int_{\mathbb{R}^3}\sum_{i,j=1}^3 gI_{ij} S_{ij}=\int_{\mathbb{R}^3} tr(S)g=0.
\end{equation}
In order to show that $\Hess(f) \in \left (L^2_{st}\right )^\perp$, we will use the property that for $S \in L^2_{st}$
\begin{equation}
tr\left ( (\nabla \otimes \nabla) S \right )= \sum_{i,j=1}^3 \partial_{x_i} \partial_{x_j} S_{ij}=0.
\end{equation}
Because $S \in L^2$ and therefore $\hat{S}\in L^2$, the above condition can be expressed as
\begin{equation}
\sum_{i,j=1}^3 \xi_i \xi_j \hat{S}_{ij}(\xi)=0,
\end{equation}
almost everywhere $\xi \in \mathbb{R}^3$.
Using the fact that the Fourier transform is an isometry on $L^2,$ and $\Hess(f), S \in L^2$ we compute that
\begin{equation}
\left < \Hess(f),S\right >= \left <\widehat{\Hess(f)},\hat{S} \right >=
-4 \pi^2 \int_{\mathbb{R}^3}\bar{\hat{f}}(\xi)\sum_{i,j=1}^3 \xi_i \xi_j \hat{S}_{ij}(\xi) d\xi=0
\end{equation}
\end{proof}
This means that as long as $u$ is sufficiently regular, $\Hess(p)$ and $-\frac{1}{4}|\omega|^2 I_3$ are in the orthogonal compliment of $L^2_{st}.$ This fact will play a key role in the new identity for enstrophy growth that we will prove in section 4.

\section{The relationship between strain and vorticity}
We have already established in \eqref{StreamReconstruction} that $u=-2 \divr (-\Delta)^{-1}S$.  The antisymmetric part of the gradient tensor then, can be reconstructed applying a zeroth order pseudo-differential operator to $S$.  We find that
\begin{equation}
A=(\nabla \otimes \nabla) (- \Delta)^{-1} S - \left ((\nabla \otimes \nabla) (- \Delta)^{-1} S \right)^*.
\end{equation}
Because this is a zeroth order operator related to the Riesz transform, it is bounded from $L^p$ to $L^p$ for $1<p<+\infty$, but we will only have Calderon-Zygmund type estimates, so our control will be very bad.  We can say something much stronger in the case of $L^2$, and in fact for every Hilbert space $\dot{H}^\alpha, -\frac{3}{2} <\alpha <\frac{3}{2}$.  
\begin{proposition}[Hilbert space isometries] \label{TensorIsometry}
For all $-\frac{3}{2} <\alpha <\frac{3}{2},$ and for all $u$ divergence free in the sense that 
$\xi \cdot \hat{u}(\xi)=0$ almost everywhere,
\begin{equation}
\|S\|_{\dot{H^\alpha}}^2=\|A\|_{\dot{H^\alpha}}^2=\frac{1}{2}\|\omega\|_{\dot{H^\alpha}}^2=\frac{1}{2}\|\nabla \otimes u\|_{\dot{H}^{\alpha}}^2.
\end{equation}
\end{proposition}
\begin{proof}
First fix $s,$ $-\frac{3}{2}<s<\frac{3}{2}.$ We will begin relating the $H^s$ norms of the anti-symmetric part and the vorticity. Recall that 
\begin{equation}
A=\frac{1}{2} \left (
\begin{matrix}
0 & \omega_3 & -\omega_2 \\
-\omega_3 & 0 & \omega_1 \\
\omega_2 & -\omega_1 & 0 \\
\end{matrix}
\right ),
\end{equation}
Therefore, for all $x \in \mathbb{R}^3,$ 
\begin{equation}
|(- \Delta)^\frac{s}{2} A(x)|^2= \frac{1}{2}|(- \Delta)^\frac{s}{2} \omega(x)|^2.
\end{equation}
Because in general we have that $\|f\|_{\dot{H}^s}=\|(- \Delta)^\frac{s}{2} f\|_{L^2},$ it immediately follows that
\begin{equation}
\|A\|_{\dot{H}^s}^2=\frac{1}{2}\|\omega\|_{\dot{H}^s}^2.
\end{equation}
Because $u$ is divergence free, in Fourier space 
\begin{align}
|\hat{\omega}(\xi)|&=|2 \pi i \xi \times \hat{u}(\xi)|\\
&= 2 \pi |\xi| |\hat{u}(\xi)|\\
&=|\widehat{\nabla \otimes u}(\xi)|.
\end{align}
From this we can conclude that 
\begin{equation}
\|\omega\|_{\dot{H}^s}^2=\|\nabla \otimes u\|_{\dot{H}^s}^2.
\end{equation}
Finally we will compute
\begin{equation}
\left|(- \Delta)^\frac{s}{2}\left (\nabla \otimes u \right )\right|^2= \tr\left(\left(- \Delta)^\frac{s}{2}S+(- \Delta)^\frac{s}{2}A\right)\left((- \Delta)^\frac{s}{2}S^*+(- \Delta)^\frac{s}{2}A^* \right)\right).
\end{equation}

However, we know that the trace of the product of a symmetric matrix and an antisymmetric matrix is always zero, so we can immediately see that
\begin{equation}
\left|(- \Delta)^\frac{s}{2}\left (\nabla \otimes u \right )\right|^2=\left|(- \Delta)^\frac{s}{2}S\right|^2+\left|(- \Delta)^\frac{s}{2}A\right|^2.
\end{equation}
From this it follows that
\begin{equation}
\|\nabla \otimes u\|_{\dot{H}^s}^2=\|S\|_{\dot{H}^s}^2+\|A\|_{\dot{H}^s}^2,
\end{equation}
but we have already established that 
\begin{equation}
\|A\|_{\dot{H}^s}^2= \frac{1}{2}\|\nabla \otimes u\|_{\dot{H}^s}^2,
\end{equation}
so we can conclude that 
\begin{equation}
\|A\|_{\dot{H}^s}^2=\|S\|_{\dot{H}^s}^2= \frac{1}{2}\|\nabla \otimes u\|_{\dot{H}^s}^2.
\end{equation}
This concludes the proof.
\end{proof}

We have now established all the necessary basics and will proceed to considering enstrophy growth in terms of the strain and vorticity equations.

\section{Enstrophy and the $L^2$ growth of the strain tensor}
Before proceeding further, however, we need to show the existence of solutions in a suitable space. Leray solutions, first developed in the classic paper \cite{Leray}, are not the most well adapted to the Navier-Stokes strain equation, so we will work with mild solutions instead \cite{Kato}. We will begin by defining mild solutions in ${\dot{H^1}}$ to the Navier-Stokes equations, and then adapt this definition for mild solutions in $L^2$ to the Navier-Stokes strain equation and the vorticity equation.

Thus far, we have only considered the Navier-Stokes equation with no external force; therefore it is now necessary to state some definitions involving an external force $f$. For the incompressible Navier-Stokes equations with a force we still require that 
$\nabla \cdot u=0,$ but now the evolution equation is given by
\begin{equation}
\partial_t u+(u \cdot \nabla )u -\Delta u +\nabla p=f.
\end{equation}
The evolution equation for the vorticity is now given by
\begin{equation}
\partial_t \omega+(u \cdot \nabla)\omega -\Delta \omega -S\omega= \nabla \times f,
\end{equation}
and the evolution equation for the strain is given by
\begin{equation}
\partial_t S+ (u \cdot \nabla )S - \Delta S+ S^2+ \frac{1}{4}\omega \otimes \omega
+\frac{1}{4}|\omega|^2I_3+Hess(p)=\nabla_{sym}f.
\end{equation}
We will define a mild solution to the Navier-Stokes equation with an external force as follows.

\begin{definition}[Mild velocity solutions]
\label{ExtForce}
Suppose $u \in C\left([0,T];\dot{H}^1\right)
\cap L^2\left([0,T];\dot{H}^2\left(
\mathbb{R}^3\right)\right).$
Then $u$ is a mild solution to the Navier-Stokes equation with external force
$f\in L^2\left([0,T];L^2\left(\mathbb{R}^3
\right)\right)$if $\nabla \cdot u=0$ in $L^2$ and 
\begin{equation}
u(\cdot,t)=e^{t \Delta}u^0
+\int_0^t e^{(t-\tau)\Delta}
\left((-u \cdot \nabla)u-\nabla p+ f\right)
(\cdot, \tau)d\tau,
\end{equation}
where $p$ is defined in terms of $u$ and $f$ by convolution with the Poisson kernel
\begin{equation}
p = (- \Delta)^{-1}\left (\sum_{i,j=1}^3 
\frac{\partial u_j}{\partial x_i} \frac{\partial u_i}{\partial x_j}
-\nabla \cdot f\right ),
\end{equation}
and where $e^{t\Delta}$ is the heat operator given by convolution with the heat kernel; that is to say, $e^{t\Delta}u^0$ is the solution of the heat equation after time $t,$ with initial data $u^0.$ 
\end{definition}

Now we can define a mild solution to the Navier-Stokes strain equation accordingly.

\begin{definition}[Mild strain solutions]
Suppose $S \in C \left([0,T];L^2_{st} \right ) \cap L^2\left ([0,T]:\dot{H}^1(\mathbb{R}^3) \right ).$ Then we will call $S$ a mild solution to the Navier-Stokes strain equation \eqref{NavierStrain} with external force $f\in L^2\left([0,T];L^2\left(\mathbb{R}^3\right)\right)$ if and only if for all $0<t \leq T,$
\begin{multline}
S(\cdot,t)= e^{t \Delta} S^0+
 \int_0^t e^{(t-\tau)\Delta} \\
\left(-(u \cdot \nabla)S-S^2-\frac{1}{4}\omega\otimes \omega +\frac{1}{4}|\omega|^2 I_3 -\Hess(p) +\nabla_{sym}f  \right)(\cdot,\tau)d\tau,
\end{multline}
where $u=- 2 \divr (-\Delta)^{-1}S,$ $\omega=\nabla \times u,$ and $p=(-\Delta)^{-1}\left(|S|^2-\frac{1}{2}|\omega|^2
-\nabla \cdot f\right)$
\end{definition}

It is a classical result that $\dot{H^1}$ mild solutions to the Navier-Stokes equation exist locally in time. We will state this result precisely and then use the result to establish local in time existence of $L^2$ mild solutions to the Navier-Stokes Strain equation.
\begin{theorem}[Mild velocity solutions exist for short times] \label{KATO}
Suppose $f=0$. Then there exists a constant $C>0,$ such that for all $u^0\in \dot{H^1}(\mathbb{R}^3), \nabla \cdot u^0=0,$ for all $0<T<\frac{C}{\|u^0\|_{\dot{H^1}}^4}$, there exists a unique mild solution to the Navier-Stokes equation $u \in C\left([0,T];\dot{H^1}(\mathbb{R}^3) \right)$. 
Furthermore for all $0<\epsilon<T$, $u \in
C\left([\epsilon,T];\dot{H}^\alpha(\mathbb{R}^3)\right )$ for all $\alpha> 1,$
and therefore $u \in C^\infty\left ((0,T]\times \mathbb{R}^3;\mathbb{R}^3\right ).$ 

In the case where $f\neq 0$ for all $u^0 \in \dot{H}^1(\mathbb{R}^3), \nabla \cdot u=0$ and all $f\in L^2_{loc}\left (
(0,T^*);L^2(\mathbb{R}^3) \right )$ there exists $0<T\leq T^*$ and $u\in C\left ([0,T];\dot{H^1}(\mathbb{R}^3) \right )
\cap L^2 \left([0,T];\dot{H}^2
(\mathbb{R}^3)\right )$ such that $u$ is a mild solution to the Navier-Stokes equation. Note that mild solutions with a nonsmooth force are not smooth in general, because the bootstrapping argument for higher regularity will not work in this case.
\end{theorem}

This is the classical result \cite{Kato} of Fujita and Kato. It has since been generalized to weaker spaces \cite{Chemin,MildSobolev}.
The result is proven using an iteration scheme. Using this result, we can easily use the properties of the heat semigroup to establish the local existence in time of $L^2$ mild solutions to the Navier-Stokes strain equation.

\begin{theorem}[Mild strain solutions exist for short times] \label{MildStrain}
Suppose $f=0.$ Then there exists a constant $C>0$ such that for all $S^0 \in L^2_{st}$, if $T<\frac{C}{4\|S^0\|_{L^2}^4},$ then there exists a unique mild solution to the Navier-Stokes Strain equation \eqref{NavierStrain} $S \in C \left([0,T];L^2_{st} \right ).$
Furthermore for all $0<\epsilon<T$, 
$S \in C\left([\epsilon,T];\dot{H}^\alpha \right )$ for all $\alpha> 0,$
and therefore $S \in C^\infty\left ((0,T]\times \mathbb{R}^3\right ).$

In the case where $f\neq 0$ for all $S^0 \in L^2_{st}$ and all $f\in L^2_{loc}\left (
(0,T^*);L^2(\mathbb{R}^3) \right )$ there exists $0<T_{max}\leq T^*$ and $S\in C\left ([0,T_{max});L^2_{st} \right )
\cap L^2 \left([0,T_{max});\dot{H}^1
(\mathbb{R}^3)\right )$ such that $S$ is a mild solution to the Navier-Stokes strain equation.
\end{theorem}
\begin{proof}
We begin by inverting the strain tensor $S^0$ to recover the initial velocity:
\begin{equation}
u^0=-2 \divr (- \Delta)^{-1}S^0.
\end{equation}
We can see from the pseudo-differential operator used to obtain $u^0,$ that $S^0 \in L^2$ implies $u^0 \in \dot{H^1}.$ This means that we can apply the theorem above to show that there exists $T>0,$ such that there is a $\dot{H^1}$ mild solution
$u\in C\left ([0,T];\dot{H^1}(\mathbb{R}^3) \right )
\cap L^2 \left([0,T];\dot{H}^2
(\mathbb{R}^3)\right )$
with $u(\cdot, 0)=u^0$, that is, for all $0<t \leq T,$  
\begin{equation}
u(\cdot,t)=e^{t\Delta}u^0
+\int_0^t e^{(t-\tau)\Delta}\left(-(u \cdot \nabla)u-\nabla p+f \right)(\cdot,\tau) d\tau.
\end{equation}
Next we will compute $S=\nabla_{sym}u.$
Clearly $S \in C \left([0,T];L^2_{st} \right ) \cap L^2\left ([0,T]:\dot{H}^1(\mathbb{R}^3) \right ).$
When taking the derivative of a convolution, the derivative can be applied to either of the functions being convolved; in this case, we will apply the differential operator $\nabla_{sym}$ to $u^0$ and $(u \cdot \nabla)u+\nabla p,$ and find that
\begin{multline}
S(\cdot,t)=e^{t \Delta}S^0
+\int_0^t e^{(t-\tau)\Delta} \\
\left(-(u \cdot \nabla)S-S^2
-\frac{1}{4}\omega\otimes\omega
+\frac{1}{4}|\omega|^2 I_3 -\Hess(p)
+\nabla_{sym}f
\right)(\cdot,\tau) d\tau .
\end{multline}
Finally the higher order regularity of a mild solution $u$ when $f=0$ proved by Kato and Fujita in \cite{Kato} immediately implies higher order regularity for
$S=\nabla_{sym} u.$
This completes the proof.
\end{proof}
Now that existence of mild solutions in a suitable space is ensured, we can use the Navier-Stokes Strain equation to simplify the identity for enstrophy growth.

\begin{theorem}[Enstrophy growth identity] \label{Strn}
Suppose $S \in C \left([0,T];L^2_{st} \right ) \cap L^2\left ([0,T]:\dot{H}^1(\mathbb{R}^3) \right )$ is a mild solution to the Navier-Stokes strain equation. Then almost everywhere $0<t\leq T,$
\begin{equation} \label{StrainGrowth}
\partial_t \|S(\cdot,t)\|_{L^2}^2=-2\|S\|_{\dot{H^1}}^2- \frac{4}{3} \int_{\mathbb{R}^3} \tr(S^3)
+\left <-\Delta u,f \right >.
\end{equation}
\end{theorem}
\begin{proof}
Using \eqref{NavierVorticity}, we can compute the rate of change of enstrophy
\begin{equation}
\partial_t \frac{1}{2}\|\omega(\cdot,t)\|_{L^2}^2 = -\left <-\Delta \omega, \omega \right >-\left <(u \cdot \nabla) \omega, \omega
\right >+\left < S\omega, \omega \right >+\left <\nabla \times f,\omega \right >.
\end{equation}
Next we can integrate by parts to show that $\left <\nabla \times f,\omega \right >=\left <f,-\Delta u \right >$ and $\left < \omega, (u \cdot \nabla)\omega \right>=0,$ using the divergence free condition in the latter case. Therefore we find that
\begin{equation} \label{Enstrophy}
\partial_t \frac{1}{2}\|\omega(\cdot,t)\|_{L^2}^2 = -\|\omega\|_{\dot{H^1}}^2+ \left < S; \omega \otimes \omega \right >+\left <-\Delta u,f\right >.
\end{equation}

This is the standard identity for enstrophy growth, based on the interaction of the Strain matrix and the vorticity. See chapter 7 in \cite{NS21} for more details. We can use the isometry in Proposition \ref{isometry} to restate \eqref{Enstrophy} in terms of strain:
%Need to fix the above. Referencing should work better%
\begin{equation} \label{Tensor1}
\partial_t \|S(\cdot,t)\|_{L^2}^2= -2\|S\|_{\dot{H^1}}^2 +\left < S; \omega \otimes \omega \right >
+\left<-\Delta u,f \right >.
\end{equation}
However we can also calculate the $L^2$ growth of the strain tensor directly from our evolution equation for the strain tensor \eqref{NavierStrain},
\begin{multline}
\partial_t \|S(\cdot,t)\|_{L^2}^2=-2 \left < -\Delta S, S \right > - 2 \left < (u \cdot \nabla )S, S \right > - 2 \left < S^2, S \right > \\ -\frac{1}{2} \left < \omega \otimes \omega; S \right >- 2 \left <\Hess(p), S \right >+
\frac{1}{2}\left<|\omega|^2 I_3, S \right>
+2\left <\nabla_{sym}f,S\right >.
\end{multline}
Integrating by parts we know that $\left < (u \cdot \nabla )S, S \right > =0$.
Note that $S \in C\left ( [0,T],L^2 \right ) \cap
L^2\left ((0,T], \dot{H}^1 \right ).$ In particular this implies that $S(\cdot,t), \omega (\cdot, t) \in L^2 \cap L^6$ almost everywhere $0<t \leq T.$  This means that $S(\cdot,t), \omega (\cdot, t) \in L^3,$ so $\left < S; \omega \otimes \omega \right >$ and $\int \tr (S^3)$ are both well defined. This also means that 
$|\omega(\cdot, t)|^2, \Hess(p)(\cdot, t) \in L^2$ for almost all $0<t\leq T.$ Therefore we can apply Proposition \ref{isometry} and find that $|\omega|^2 I_3, \Hess(p) \in \left (L^2_{st} \right )^ \perp$, so
\begin{align}
\left <\frac{1}{2}|\omega|^2I_3, S \right >&=0,\\
\left < \Hess(p), S \right >&= 0.
\end{align}

Now we can use the fact that $S$ is symmetric to compute that
\begin{equation}
    \left<S^2,S\right>=
    \int_{\mathbb{R}^3}\tr(S^3),
\end{equation}
and that 
\begin{align}2\left <\nabla_{sym}f,S \right >&=
2\left<\nabla \otimes f,S\right >\\
&=\left <f, -2\divr(S)\right >\\
&=\left <f, -\Delta u \right >.
\end{align}
Putting all of these together we find that
\begin{equation} \label{Tensor2}
\partial_t \|S(\cdot,t)\|_{L^2}^2=-2\|S\|_{\dot{H^1}}^2 -\frac{1}{2}\left< S; \omega \otimes \omega \right> - 2\int_{\mathbb{R}^3}\tr(S^3)
+\left <-\Delta u,f \right>.
\end{equation}
We now will add $\frac{1}{3}$ \eqref{Tensor1} to  $\frac{2}{3}$ \eqref{Tensor2} to cancel the term $\left< S, \omega \otimes \omega \right >$, and we find
\begin{equation}
\partial_t \|S(\cdot,t)\|_{L^2}^2=-2\|S\|_{\dot{H^1}}^2- \frac{4}{3}\int_{\mathbb{R}^3}\tr(S^3)
+\left <-\Delta u,f \right>.
\end{equation}

Finally we will note that because the subcritical quantity $\|S(\cdot,t)\|_{L^2}$ is controlled uniformly on $[0,T],$ the smoothing due to the heat kernel guarantees that $S$ is smooth when $f=0$, so the identity \eqref{StrainGrowth} can be understood as a derivative of a smooth quantity in the classical sense in this case. When $f\neq 0,$
we cannot assume that $S\in C^1\left([0,T];L^2\left(\mathbb{R}^3 \right)\right),$ which is why we can only say the identity holds almost everywhere in the general case with a nonzero external force.
\end{proof}

Now that we have improved the estimate for enstrophy growth from one that involved the interaction of the vorticity and the strain tensor to an estimate that only involves the strain tensor. We can still extract more geometric information about the flow, however. The identity for enstrophy growth in Theorem \ref{Strn} can also be expressed in terms of $\det(S).$
\begin{corollary}[Alternative enstrophy growth identity] \label{DET2}
Suppose
$S \in C \left([0,T];L^2_{st} \right ) \cap L^2\left ([0,T]:\dot{H}^1(\mathbb{R}^3) \right )$ is a mild solution to the Navier-Stokes strain equation. Then for almost all $0 < t\leq T,$
\begin{equation} \label{VortDet}
\partial_t \|S(\cdot,t)\|_{L^2}^2=-2\|S\|_{\dot{H^1}}^2- 4 \int_{\mathbb{R}^3} \det(S)+\left <-\Delta u,f \right>.
\end{equation}
\end{corollary}

\begin{proof}
Because $S$ is symmetric it will be diagonalizable with three real eigenvalues, and because $S$ is trace free, we have $\tr(S)=\lambda_1+\lambda_2+\lambda_3=0$.  This allows us to relate $\tr(S^3)$ to $\det(S)$ by
\begin{align}
\tr(S^3)&= \lambda_1^3+\lambda_2^3+\lambda_3^3\\
&=\lambda_1^3+\lambda_2^3 + (-\lambda_1- \lambda_2)^3\\
&= -3\lambda_1^2 \lambda_2- 3 \lambda_1 \lambda_2^2\\
&= 3(-\lambda_1 -\lambda_2 )\lambda_1 \lambda_2\\ 
&= 3 \lambda_1 \lambda_2 \lambda_3\\
&=3  \det(S).
\end{align}
Therefore we can write our identity for enstrophy growth as:
\begin{equation}
\partial_t \|S(\cdot,t)\|_{L^2}^2=-2\|S\|_{\dot{H^1}}^2- 4 \int_{\mathbb{R}^3} \det(S)+\left <-\Delta u,f \right>.
\end{equation}
This completes the proof.
\end{proof}

\begin{remark}
    As mentioned in the introduction, Dongho Chae proved the analogous result, 
    \begin{equation}\partial_t \|S(\cdot,t)\|_{L^2}^2=- 4 \int_{\mathbb{R}^3} \det(S), \end{equation}
    in the context of smooth solutions to the Euler equation with no external force \cite{ChaeStrain}. In this paper he shows directly that
    \begin{equation} \label{ChaeIdentity}
    \partial_t \frac{1}{2}\|\nabla \otimes u(\cdot,t)\|_{L^2}^2=
    \left <\left (u\cdot \nabla \right )u,
    \Delta u\right >=-\int_{\mathbb{R}^3} \tr(S^3)+
    \frac{1}{4} \left <S; \omega \otimes
    \omega \right >.
    \end{equation}
    In the context of the Euler equation, the familiar estimate for enstrophy growth following from the vorticity equation is
    \begin{equation} \label{ChaeVort}
     \partial_t \frac{1}{2}\|\nabla \otimes u(\cdot,t)\|_{L^2}^2= \partial_t \frac{1}{2}\|\omega(\cdot,t)\|_{L^2}^2=
     \left <S; \omega \otimes \omega \right >.
    \end{equation}
    Adding $\frac{4}{3}$ \eqref{ChaeIdentity} and $-\frac{1}{3}$ \eqref{ChaeVort}, it follows that
    \begin{equation}
     \partial_t \|S(\cdot,t)\|_{L^2}^2=
     \partial_t \frac{1}{2}\|\nabla \otimes u(\cdot,t)\|_{L^2}^2=
     -\frac{4}{3}\int_{\mathbb{R}^3} \tr(S^3)=
     -4\int_{\mathbb{R}^3} \det(S).
    \end{equation}
\end{remark}

The identity for enstrophy growth in Corollary \ref{DET2} gives us a significantly better understanding of enstrophy production than the classical enstrophy growth identity \eqref{Enstrophy}, because we now have the growth controlled solely in terms of the strain tensor, rather than both the strain tensor and the vorticity.
This estimate also provides analytical confirmation of the well known result that the vorticity tends to align with the eigenvector corresponding to the intermediate eigenvalue of the strain matrix \cite{VorticityAlignmentResults,ConcentratedVorticity}. Comparing the identities in \eqref{StrainGrowth}, \eqref{Enstrophy}, and \eqref{VortDet} we see that
\begin{equation} \label{VortStretch}
\left < S, \omega \otimes \omega \right > = -4 \int_{\mathbb{R}^3} \det(S)= 
-\frac{4}{3}\int \tr(S^3).
\end{equation}
When $\det(S)$ tends to be positive, it means there are two negative eigenvalues and one positive eigenvalue, so $\left < S, \omega \otimes \omega \right >$ being negative means the vorticity tends to align, on average when integrating over the whole space, with the negative eigenspaces.  Likewise, when $\det(S)$ tends to be negative, it means there are two positive eigenvalues and one negative eigenvalue, so $\left < S, \omega \otimes \omega \right >$ being positive means the vorticity tends to align, on average when integrating over the whole space, with the positive eigenspaces.  When $\det(S)$ tends to be zero when integrated over the whole space, the vorticity tends clearly to be aligned with the intermediate eigenvalue, as well. Growth in all cases geometrically corresponds to the strain matrix $S$ stretching in two directions, while strongly contracting in the third direction.

Finally we will bound the growth rate of enstrophy in terms of the size of the strain matrix, and see what this matrix looks like in the sharp case of this bound.
\begin{proposition}[Determinant bound] \label{DET} 
Let $M$ be a three by three, symmetric, trace free matrix, then
\begin{equation} \label{DeterminantBound}
-4 \det(M) \leq \frac{2}{9}\sqrt{6}|M|^3,
\end{equation}
with equality if and only if $- \frac{1}{2}\lambda_1=\lambda_2=\lambda_3,$ where $\lambda_1 \leq \lambda_2 \leq \lambda_3$ are the eigenvalues of $M$.
\end{proposition}
\begin{proof}
In the case where $M=0$, it holds trivially.  In the case where $M\neq 0$, then we have $\lambda_1<0, \lambda_3>0$.  This allows us to define a parameter $r=-\frac{\lambda_1}{\lambda_3}$. The two parameters $\lambda_3$ and $r$ completely define the system because $\lambda_1=-r \lambda_3$ and $\lambda_2=-\lambda_1-\lambda_3= (r-1)\lambda_3$. We must now say something about the range of values the parameter $r$ can take on.  $\lambda_1 \leq \lambda_2 \leq \lambda_3$ implies that
$-r \leq r-1 \leq 1$, so therefore $\frac{1}{2} \leq r \leq 2.$  Now we can observe that 
\begin{equation}
-4\det(M)=-4\lambda_1 \lambda_2 \lambda_3=4r(r-1)\lambda_3^3,
\end{equation}
and that
\begin{equation}
|M|^2=\lambda_1^2+\lambda_2^2+\lambda_3^2=(r^2+(r-1)^2+1)\lambda_3^2=(2r^2-2r+2)\lambda_3^2.
\end{equation}
We can combine the two equations above to find that 
\begin{equation}
-4\det(M)= \sqrt{2}\frac{r^2-r}{(r^2-r+1)^\frac{3}{2}}|M|^3.
\end{equation}
Next we will observe that
\begin{equation}
\sqrt{2}\frac{r^2-r}{(r^2-r+1)^\frac{3}{2}}\bigg|_{r=2}=\sqrt{2}\frac{2}{3 \sqrt{3}}=\frac{2}{9}\sqrt{6}.
\end{equation}
This is exactly as we want, as $r=2$ is the case that we want to correspond to equality. Finally we observe that for all $\frac{1}{2}\leq r <2$, we have that
\begin{equation}
\sqrt{2}\frac{r^2-r}{(r^2-r+1)^\frac{3}{2}}<\frac{2}{9}\sqrt{6}.
\end{equation}
This completes the proof.
\end{proof}
The structure of the quadratic term in relation to $r=-\frac{\lambda_1}{\lambda_3}=2,$ the extremal case, will be investigated further in section 6 when we consider blow up for a toy model ODE for the Navier-Stokes strain equation. It is an interesting open question whether or not there is a strain matrix which saturates this inequality globally in space. More precisely, is there an $S \in L^2_{st},$ not identically zero, such that 
$\lambda_2(x)=\lambda_3(x)$ almost everywhere $x \in \mathbb{R}^3$?
\begin{corollary} Suppose $S \in C \left([0,T];L^2_{st} \right ) \cap L^2\left ([0,T]:\dot{H}^1(\mathbb{R}^3) \right )$ is a mild solution to the Navier-Stokes strain equation with $f=0.$ Then for all $0<t\leq T,$
\begin{equation}
\partial_t \|S(\cdot,t)\|_{L^2}^2 \leq -2\|S\|_{\dot{H^1}}^2+ \frac{2}{9}\sqrt{6}\int_{\mathbb{R}^3}|S|^3.
\end{equation}
\end{corollary}
\begin{proof}
This corollary follows immediately from Proposition \ref{DET} and Corollary \ref{DET2}.
\end{proof}

This corollary allows us to derive an a priori estimate on the growth of enstrophy, which will then give us a lower bound on $T_{max}$.
\begin{theorem}[Enstrophy Growth Differential Inequality] \label{MaxGrowth}
Suppose $S \in C \left([0,T];L^2_{st} \right ) \cap L^2\left ([0,T]:\dot{H}^1(\mathbb{R}^3) \right )$ is a mild solution to the Navier-Stokes strain equation with $f=0$. Then for all $0<t \leq T,$
\begin{equation} \label{enstrophy derivative}
\partial_t \|S(\cdot, t)\|_2^2 \leq \frac{1}{1458 \pi^4} \|S(\cdot, t)\|_2^6,
\end{equation}

\begin{equation} 
\|S(\cdot, t)\|_2^2 \leq \left ( \frac{1}{\frac{1}{\|S^0\|_{L^2}^4}- \frac{1}{729 \pi^4}  t}\right )^{\frac{1}{2}}.
\end{equation}
\end{theorem}
\begin{proof}
For the first portion \eqref{enstrophy derivative} we begin by applying the interpolation inequality for $L^p$ to bound the $L^3$ norm by the $L^2$ and $L^6$ norms:
\begin{equation}
\int_{\mathbb{R}^3}|S|^3 \leq  \|S\|_2^{\frac{3}{2}}\|S\|_6^{\frac{3}{2}}.
\end{equation}

The sharp Sobolev inequality \cite{Talenti,SharpSobolev} states that:
\begin{equation}
\|S\|_{L^6} \leq  \left (\frac{1}{3^{\frac{3}{4}}}
\frac{2}{\pi} \right )^{\frac{2}{3}} \|S\|_{\dot{H^1}}.
\end{equation}
Observing that $\left (\frac{2}{3} \right )^{\frac{3}{2}}=\frac{2}{9}\sqrt{6},$ we can combine the above equations to find that:
\begin{equation}
\partial_t \|S(\cdot,t)\|_{L^2}^2 \leq -2\|S\|_{\dot{H^1}}^2 +
\left ( \frac{2}{3} \right )^{\frac{3}{2}} \frac{1}{3^{\frac{3}{4}}} \frac{2}{\pi}  \|S\|_{\dot{H^1}}^\frac{3}{2}\|S\|_{L^2}^\frac{3}{2}.
\end{equation}
We will regroup terms and write this as
\begin{equation}
\left ( \frac{2}{3} \right )^{\frac{3}{2}} \frac{1}{3^{\frac{3}{4}}} \frac{2}{\pi}  \|S\|_{\dot{H^1}}^\frac{3}{2}\|S\|_{L^2}^\frac{3}{2}=\left ( \left (\frac{8}{3} \right )^{\frac{3}{4}} \|S\|_{\dot{H^1}}^\frac{3}{2} \right ) 
\left (\frac{2^\frac{5}{2}}{2^\frac{9}{4}} \frac{1}{3^\frac{3}{4} \pi} \|S\|_{L^2}^\frac{3}{2} \right ) 
\end{equation}
We now apply Young's inequality using the conjugate exponents $\frac{4}{3}$ and $4$ to show that
\begin{equation}
\left ( \left (\frac{8}{3} \right )^{\frac{3}{4}} \|S\|_{\dot{H^1}}^\frac{3}{2} \right ) 
\left (\frac{2^\frac{5}{2}}{2^\frac{9}{4}} \frac{1}{3^\frac{3}{2} \pi} \|S\|_{L^2}^\frac{3}{2} \right )   \leq
\frac{3}{4}\left ( \left (\frac{8}{3} \right )^{\frac{3}{4}} \|S\|_{\dot{H^1}}^\frac{3}{2} \right )^\frac{4}{3}
+\frac{1}{4}\left (2^\frac{1}{4}\frac{1}{3^\frac{3}{2} \pi} \|S\|_{L^2}^\frac{3}{2} \right )^4
\end{equation}
Therefore we can conclude that
\begin{equation}
\partial_t \|S(\cdot, t)\|_{L^2}^2 \leq 
\frac{1}{1458 \pi^4}\|S(\cdot,t)\|_{L^2}^6
\end{equation}
This completes the proof of the first part of the theorem \eqref{enstrophy derivative}; the second portion of the theorem follows immediately from integrating this differential inequality.
\end{proof}

This is a significant improvement on the best known estimates for enstrophy growth. If we take the enstrophy to be one half of the square of the $L^2$ norm of the vorticity, $E(t)=\frac{1}{2}\|\omega(\cdot,t)\|_{L^2}^2=\|S(\cdot,t)\|_{L^2}^2,$ then this means
\begin{equation}
\partial_t E(t) \leq \frac{1}{1458 \pi^4} E(t)^3.
\end{equation}
The previous best known estimate for enstrophy growth
\cite{EnstrophyGrowth,EnstrophyGrowth2,Protas} was 
\begin{equation}
\partial_t E(t) \leq \frac{27}{8 \pi^4}E(t)^3.
\end{equation}
This work has recently been extended in \cite{Protas}, which shows numerically that solutions which are sharp for this inequality locally in time, actually tend to decay fairly quickly, and so are not good candidates for blowup. This work also suggested numerically that the constant $\frac{27}{8 \pi^4}$ is non-optimal, but we will note that this work was on the torus, which may result in a different constant than the whole space as, for example, the sharp Sobolev constant may not be the same. Nonetheless, the fairly drastic improvement in the constant is not too surprising, because $\frac{27}{8 \pi^4}$ was not expected to be sharp.

Previous papers investigating the growth of enstrophy, such as in \cite{EnstrophyGrowth2} simply apply the sharp  Gagliardo-Nirenberg inequality to bound $\int_{\mathbb{R}^3} -\Delta u \cdot \left ( (u \cdot \nabla)u \right )$. A finer analysis of the role of the strain matrix, using both the vorticity equation and the strain equation, allows us to drastically improve this estimate by a factor of $4{,}920.75.$
This is a huge quantitative improvement on the estimate for enstrophy growth, although the estimate is of course still cubic, so an improvement in the constant only increases the minimum time until a solution might blow up, it cannot rule out blowup.
In fact, Theorem \ref{MaxGrowth} provides almost immediately as a corollary an estimate for the quickest possible blowup time in terms of the initial enstrophy. First, we must define $T_{max},$ the maximal time of existence for a mild solution corresponding to some initial data $S^0 \in L^2_{st}.$
\begin{definition}[Maximal time of existence] \label{Tmax}
Suppose $S^0 \in L^2_{st}$ and 
$f\in L^2_{loc}\left ([0,T^*)L^2(\mathbb{R}^3) \right )$
Then the maximal time of existence for a smooth solution is
\begin{equation}
T_{max}=\sup \left \{ 0<T\leq T^*: S \in C \left([0,T],L^2_{st}
\right )\cap L^2 \left ([0,T]:\dot{H}^1(\mathbb{R}^3) \right ),
 S(\cdot,0)=S^0 \right \},
\end{equation}
where $S$ is a mild solution to the Navier-Stokes strain equation with initial data $S^0$ and external force $f.$
\end{definition}
We will note here that $T_{max}$ for a mild solution to the Navier-Stokes strain equation with initial data $S^0 \in L^2_{st}$ is equivalent to $T_{max}$ for a mild solution to the Navier-Stokes equation corresponding to initial data 
$u^0=-2 \divr (-\Delta)^{-1}S^0 \in \dot{H}^1.$ We will also note that when $f=0, T_{max}=+\infty$ corresponds to a global smooth solution, whereas $T_{max}<+\infty$ corresponds to a solution that develops singularities in finite time.
Whether or not there exist smooth solutions to the Navier-Stokes equation that develop singularities in finite time is one of the biggest open problems in partial differential equations, and is one of the Millennium Problems put forward by the Clay Mathematics Institute \cite{Clay}.
Definition \ref{Tmax} is directly related to this problem; the Millennium Problem could be stated equivalently as: show $T_{max}=+\infty$ 
for all $S^0\in L^2_{st}$ when $f=0$ or provide a counterexample.

We will now prove a lower bound on $T_{max}$ based on the growth estimate in Theorem \ref{MaxGrowth}. 
\begin{corollary}[Lower bound on time to blowup]
Suppose $S^0 \in L^2_{st}$ and $f=0.$ Then 
\begin{equation}
T_{max} \geq \frac{729 \pi^4}{\|S^0\|_{L^2}^4}.
\end{equation}
That is, for all $T<\frac{729 \pi^4}{\|S^0\|_{L^2}^4}$ there exists a mild solution $S \in C \left([0,T];L^2_{st} \right ) \cap L^2\left ([0,T]:\dot{H}^1(\mathbb{R}^3) \right )$ to the Navier-Stokes strain equation with initial data $S^0$, and this solution is smooth.
\end{corollary}
\begin{proof}
Fix $T<\frac{729 \pi^4}{\|S^0\|_{L^2}^4}.$ In Theorem \ref{MildStrain} we showed that mild solutions must exist locally in time for any initial data $S^0\in L^2_{st}.$ From this it follows that unless $\|S(\cdot,t)\|_{L^2}$ becomes unbounded as $t \to \tau,$ then a mild solution $S \in C\left ((0,\tau);L^2_{st} \right )$ can be extended beyond $\tau$ to some $\tau'>\tau.$ It is clear from
Theorem \ref{MaxGrowth} that $\|S(\cdot,t)\|_{L^2}$ remains bounded for all $t\leq T<\frac{729 \pi^4}{\|S^0\|_{L^2}^4},$ so this establishes the existence of a mild solution $S \in C\left ([0,T];L^2_{st} \right )$. 
In particular, this implies that $\|S(\cdot,t)\|_{L^2}$ must become unbounded as $t \to T_{max}$ if $T_{max}<+\infty.$
We already showed in Theorem \ref{MildStrain} that mild solutions are smooth when $f=0$, so this completes the proof.
\end{proof}

Note that this is similar to the estimate in the initial theorem establishing the local in time existence of mild solutions; the only difference is that we have improved the constant. The estimate on local existence in \cite{MildSobolev} does not actually state the value of the constant $C>0.$
It is shown in \cite{Protas} that
\begin{equation}
T_{max}\geq \frac{4 \pi^4}{27\|S^0\|_{L^2}^4},
\end{equation}
although their statement is in terms of $E_0=\frac{1}{2}\|\omega^0\|_{L^2}^2,$  as their analysis does not focus on the strain.
The detailed analysis of the evolution of strain itself allows us to significantly sharpen the lower bounds on the minimal blowup time by a factor of $4{,}920.75$ from previous estimates, which were derived using standard harmonic analysis methods, and did not take full advantage of the structures that incompressible flow imposes on the Navier-Stokes problem, for instance that the strain matrix must be trace free.

Now that we have finished outlining the main estimates that can be derived from the Navier-Stokes strain equation, we will go on to use these estimates to prove a new regularity criterion in terms of the middle eigenvalue of the strain matrix. We will then consider a toy model ODE that captures some of the features of the quadratic term and tells us a little bit about the local structure of blow-up solutions, in particular what the distribution of eigenvalues will tend towards.

\section{Regularity criteria for the eigenvalues of the strain}
In this section we will prove the main result of the paper, Theorem \ref{IntroEigenRegCrit}, as well as some immediate corollaries that were also stated in the introduction.
Before we can prove that regularity criteria, we will need to prove a lemma bounding the growth of enstrophy in terms of $\lambda_2^+$.
\begin{lemma}[Middle eigenvalue determinant bound] \label{EigenBound}
Suppose $S \in C \left([0,T];L^2_{st} \right ) \cap L^2\left ([0,T]:\dot{H}^1(\mathbb{R}^3) \right )$ is a mild solution to the Navier-Stokes strain equation with external force $f\in L^2\left ((0,T^*);L^2\left (\mathbb{R}^3\right)\right )$. and $S(x)$ has eigenvalues $\lambda_1(x) \leq \lambda_2(x) \leq \lambda_3 (x).$ Define 
\begin{equation}
\lambda_2^+(x)=\max \{\lambda_2(x),0\}.
\end{equation}
Then
\begin{equation}
-\det(S) \leq \frac{1}{2} |S|^2 \lambda_2^+.
\end{equation}
and for almost all $0<t\leq T,$
\begin{equation} \label{gcon}
\partial_{t}\|S(\cdot,t)\|_{L^2}^2 \leq 
-\|S\|_{\dot{H}^1}^2
+2 \int_{\mathbb{R}^3} \lambda_2^+|S|^2
+\frac{1}{2}\|f\|_{L^2}^2.
\end{equation}
\end{lemma}
\begin{proof}
We will begin by noting that $\lambda_1 \leq 0$ and $\lambda_3 \geq 0,$ so clearly, $-\lambda_1 \lambda_3 \geq 0.$ This implies that
\begin{align}
-\det(S)&=(-\lambda_1 \lambda_3) \lambda_2\\
&\leq (-\lambda_1 \lambda_3) \lambda_2^+.
\end{align}
Next we can apply Young's Inequality to show that
\begin{align}
-\lambda_1 \lambda_3 &\leq 
\frac{1}{2}(\lambda_1^2+ \lambda_3^2)\\
&\leq \frac{1}{2}
(\lambda_1^2+\lambda_2^2+\lambda_3^2)\\
&=\frac{1}{2} |S|^2.
\end{align}
We can combine these inequalities and conclude that
\begin{equation}
-\det(S) \leq \frac{1}{2} |S|^2 \lambda_2^+.
\end{equation}
Next we apply H\"older's inequality, Proposition \ref{isometry}, and Young's inequality to conclude that
\begin{align}
    \left <-\Delta u,f\right > &\leq
    \|-\Delta u\|_{L^2}\|f\|_{L^2}\\
    &=\sqrt{2}\|S\|_{\dot{H}^1}\|f\|_{L^2}\\
    &\leq \|S\|_{\dot{H}^1}^2
    +\frac{1}{2}\|f\|_{L^2}^2.
\end{align}
Recall from Corollary \ref{DET2}, that 
\begin{equation}
    \partial_t\|S\|_{L^2}^2=-2\|S\|_{\dot{H}^1}^2-4\int \det (S)+\left <-\Delta u,f \right >,
\end{equation}
and this completes the proof.
\end{proof}
With this bound, we are now ready to prove the main result of the paper. This is Theorem \ref{IntroEigenRegCrit} from the introduction, which is restated here for the reader's convenience.

\begin{theorem}[Middle eigenvalue of strain characterizes the blow-up time] \label{RegCrit}
Let $u \in C\left([0,T];\dot{H}^1\left (\mathbb{R}^3\right ) \right )
\cap L^2\left ([0,T];\dot{H}^2\left (\mathbb{R}^3\right )\right ),$ for all $T<T_{max}$ be a mild solution to the Navier-Stokes equation with force 
$f\in L^2_{loc}\left ((0,T^*);L^2\left (\mathbb{R}^3\right)\right )$.
If $\frac{2}{p}+\frac{3}{q}=2,$ with $\frac{3}{2}<q \leq + \infty,$ then
\begin{equation} \label{gronwall}
\|u(\cdot,T)\|_{\dot{H}^1}^2 \leq 
\left(\left\|u^0\right\|_{\dot{H}^1}^2 + \int_0^T \|f(\cdot,t)\|_{L^2}^2 dt \right)
\exp \left (C_q \int_{0}^T\|\lambda_2^+(\cdot,t)\|_{L^q\left (\mathbb{R}^3\right )}^p dt \right ),
\end{equation}
with the constant $C_q$ depending only on $q.$
In particular if the maximal existence time for a mild solution $T_{max}<T^*,$ then 
\begin{equation}
\int_{0}^{T_{max}}\|\lambda_2^+(\cdot,t)\|_{L^q(\mathbb{R}^3)}^p dt= + \infty.
\end{equation}
\end{theorem}

\begin{proof}
First we will note that $\|u(\cdot,t)\|_{\dot{H}^1}^2$ must become unbounded as $t\to T_{max}$ if the mild solution cannot be extended beyond some time $T_{max}<T^*,$ so it suffices to prove the bound \eqref{gronwall}.
Applying Proposition \ref{TensorIsometry}, it is equivalent to show that
\begin{equation}
\|S(\cdot,T)\|_{L^2}^2 \leq\left(\left\|S^0\right\|_{L^2}^2 +\frac{1}{2}\int_0^T \|f(\cdot,t)\|_{L^2}^2 dt\right )
\exp \left ( C_q \int_{0}^T\|\lambda_2^+(\cdot,t)\|_{L^q\left (\mathbb{R}^3\right )}^p dt\right ).
\end{equation}
To begin we recall the conclusion in Lemma \ref{EigenBound}
\begin{equation}
\partial_{t}\|S(\cdot,t)\|_{L^2}^2 \leq -\|S\|_{\dot{H}^1}^2
+2 \int_{\mathbb{R}^3} \lambda_2^+|S|^2+\frac{1}{2}\|f\|_{L^2}^2.
\end{equation}

First we will consider the case $q=+\infty.$
Applying H\"older's inequality with exponents $1$ and $+ \infty$ we see that,
\begin{equation}
\partial_{t}\|S(\cdot,t)\|_{L^2}^2 \leq
2\|\lambda_2^+\|_{L^\infty}\|S\|_{L^2}^2
+\frac{1}{2}\|f\|_{L^2}^2.
\end{equation}
Now we can apply Gronwall's inequality and find that
\begin{equation}
\|S(\cdot,T)\|_{L^2}^2 \leq
\left (\left\|S^0\right\|_{L^2}^2 +\frac{1}{2}\int_0^T \|f(\cdot,t)\|_{L^2}^2 dt\right )
\exp \left (2 \int_0^T\|\lambda_2^+\|_{L^\infty} dt\right ).
\end{equation}
Now we will consider the case $\frac{3}{2}<q<+\infty.$
We will begin by applying H\"older's inequality to \eqref{gcon}, 
taking $\frac{1}{q}+\frac{1}{a}=1,$
\begin{equation}
\partial_{t}\|S(\cdot,t)\|_{L^2}^2 \leq -\|S\|_{\dot{H}^1}^2
+2\|\lambda_2^+\|_{L^q}\|S\|_{L^{2a}}^2
+\frac{1}{2}\|f\|_{L^2}^2.
\end{equation}
Applying the Sobolev inequality we find
\begin{equation}
\partial_{t}\|S(\cdot,t)\|_{L^2}^2 \leq -C\|S\|_{L^6}^2
+2\|\lambda_2^+\|_{L^q}\|S\|_{L^{2a}}^2
+\frac{1}{2}\|f\|_{L^2}^2.
\end{equation}
Noting that $q>\frac{3}{2},$ it follows that $a<3,$ so $2a<6.$ Take 
$\sigma \in (0,1),$ such that $\frac{1}{2a}=\sigma \frac{1}{2}+ 
(1-\sigma)\frac{1}{6}.$ Then interpolating between $L^2$ and $L^6$ we find that
\begin{equation}
\partial_{t}\|S(\cdot,t)\|_{L^2}^2 \leq -C\|S\|_{L^6}^2
+2\|\lambda_2^+\|_{L^q} \|S\|_{L^2}^{2\sigma}\|S\|_{L^6}^{2(1-\sigma)}
+\frac{1}{2}\|f\|_{L^2}^2.
\end{equation}
We know that $\frac{\sigma}{3}+\frac{1}{6}=\frac{1}{2a},$ so 
$\sigma=\frac{3}{2a}-\frac{1}{2}.$ $\frac{1}{a}=1-\frac{1}{q},$ so 
$\sigma=1-\frac{3}{2q}.$ Therefore we conclude that
\begin{equation}
\partial_{t}\|S(\cdot,t)\|_{L^2}^2 \leq -C\|S\|_{L^6}^2
+2\|\lambda_2^+\|_{L^q} \|S\|_{L^2}^{2-\frac{3}{q}}\|S\|_{L^6}^{\frac{3}{q}}
+\frac{1}{2}\|f\|_{L^2}^2.
\end{equation}
Now take $b=\frac{2q}{3}.$ That means $1<b<+\infty.$ Define $p$ by 
$\frac{1}{p}+\frac{1}{b}=1,$ and apply Young's inequality with exponents $p$ and $b,$ and we find that
\begin{equation}
\partial_{t}\|S(\cdot,t)\|_{L^2}^2 \leq -C\|S\|_{L^6}^2
+C_q\left (\|\lambda_2^+\|_{L^q} \|S\\|_{L^2}^{2-\frac{3}{q}}\right )^p+C\|S\|_{L^6}^{b\frac{3}{q}}
+\frac{1}{2}\|f\|_{L^2}^2.
\end{equation}
Note that $\frac{1}{p}=1-\frac{1}{b}=1-\frac{3}{2q}$. This means that $p(2-\frac{3}{q})=2$ and that $\frac{2}{p}+\frac{3}{q}=2,$ and we know by definition that $b\frac{3}{q}=2,$ so 
\begin{equation}
\partial_{t}\|S(\cdot,t)\|_{L^2}^2 \leq C_q \|\lambda_2^+\|_{L^q}^p \|S\|_{L^2}^2
+\frac{1}{2}\|f\|_{L^2}^2.
\end{equation}
Applying Gronwall's inequality we find that
\begin{equation}
\|S(\cdot,T)\|_{L^2}^2 \leq \left ( \left\|S^0\right\|_{L^2}^2
+\frac{1}{2}\int_0^T\|f\|_{L^2}^2 dt \right )
\exp \left (C_q \int_{0}^T\|\lambda_2^+\|_{L^q\left (\mathbb{R}^3\right )}^p dt\right).
\end{equation}
This completes the proof.
\end{proof}

We will note here that the case $p=1, q=+\infty$ corresponds to the Beale-Kato-Majda criterion, so it may be possible to show that in this case the regularity criterion holds for the Euler equations as well as the Navier-Stokes equations. Note in particular that we did not use the dissipation to control the enstrophy, so there is a natural path to extend the result to solutions of the Euler equation as well. There is more work to do however, as bounded enstrophy is not sufficient to guarantee regularity for solutions to the Euler equations.

There is also an open question at the other boundary case, $p=+\infty$ $q=\frac{3}{2}.$ This would likely be quite difficult as the methods used in \cite{SereginL3,Gallagher} to extend the Prodi-Serrin-Ladyzhenskaya regularity criterion to the boundary case $p=+\infty,$ $q=3$ were much more technical than the methods in \cite{Ladyzhenskaya,Prodi,Serrin}. In particular, when $p=+\infty$ it is no longer adequate to rely on the relevant Sobolev embeddings, because we cannot apply Gronwall's inequality. Nonetheless, it is natural to suspect based on Theorem \ref{RegCrit} that if $u$ is a smooth solution to the Navier-Stokes equation with a maximal time of existence, $T_{max}<+\infty,$ then
\begin{equation}
\limsup_{t \to T_{max}} \|\lambda_2^+(\cdot,t)\|_{L^\frac{3}{2}}=+\infty.
\end{equation}

While we cannot prove this result, we can prove the following weaker statement.
\begin{theorem}[Regularity criterion in the borderline case]
Let $u \in C\left([0,T];\dot{H}^1\left (\mathbb{R}^3\right ) \right )
\cap L^2\left ([0,T];\dot{H}^2\left (\mathbb{R}^3\right )\right ),$ for all $T<T_{max}$ be a mild solution to the Navier-Stokes equation with force 
$f\in L^2_{loc}\left ((0,T^*);L^2\left (\mathbb{R}^3\right)\right )$.
If $T_{max}<T^*,$ then
\begin{equation}
    \limsup_{t \to T_{max}}
    \|\lambda_2^+(\cdot,t)\|_{L^\frac{3}{2}}
    \geq \frac{1}{C_s^2},
\end{equation}
where $C_s=\frac{1}{\sqrt{3}}\left(
    \frac{2}{\pi}\right)^\frac{2}{3}$ is the constant in the sharp Sobolev inequality that controls the embedding $\dot{H}^1\left(\mathbb{R}^3\right) \subset 
    L^6\left(\mathbb{R}^3\right).$
\end{theorem}
\begin{proof}
Suppose toward contradiction that $T_{max}<T^*$ and
\begin{equation}
    \limsup_{t \to T_{max}}
    \|\lambda_2^+(\cdot,t)\|_{L^\frac{3}{2}}
    < \frac{1}{C_s^2}.
\end{equation}
Then there must exist $\epsilon, \delta>0,$ such that for all $T_{max}-\delta<t<T_{max},$ 
\begin{equation}
C_s^2 \|\lambda_2^+(\cdot,t)\|_{L^\frac{3}{2}}
<1-\epsilon.
\end{equation}
Recall from the proof of Lemma \ref{EigenBound} that 
\begin{align}
    \partial_t \|S(\cdot,t)\|_{L^2}^2&\leq
    -2\|S\|_{\dot{H}^1}^2+ 
    2\int_{\mathbb{R}^3}\lambda_2^+|S|^2
    +\sqrt{2}\|S\|_{\dot{H}^1}\|f\|_{L^2}\\
    &\leq 
    -2 \|S\|_{\dot{H}^1}^2+ 
    2\|\lambda_2^+\|_{L^\frac{3}{2}}\|S\|_{L^6}^2
    +\sqrt{2}\|S\|_{\dot{H}^1}\|f\|_{L^2}\\
    &\leq
    -2 \|S\|_{\dot{H}^1}^2+ 
    2C_s^2\|\lambda_2^+\|_{L^\frac{3}{2}}
    \|S\|_{\dot{H}^1}^2
    +\sqrt{2}\|S\|_{\dot{H}^1}\|f\|_{L^2},
\end{align}
where we have applied H\"older's inequality and the sharp Sobolev inequality.

Next we recall that by hypothesis, for all $T_{max}-\delta<t<T_{max},$ 
\begin{equation}
C_s^2\|\lambda_2^+ \|_{L^\frac{3}{2}}-1<-\epsilon.
\end{equation}
Using this fact and applying Young's inequality, we find
\begin{align}
    \partial_t \|S(\cdot,t)\|_{L^2}^2&\leq
    -2 \epsilon\|S\|_{\dot{H}^1}^2+
    \sqrt{2}\|S\|_{\dot{H}^1}\|f\|_{L^2}\\
    &\leq 
    \frac{1}{4 \epsilon}\|f\|_{L^2}^2.
\end{align}

Integrating this differential inequality we find that
\begin{equation}
    \limsup_{t \to T_{max}}\|S(\cdot,t)\|_{L^2}^2\leq
    \|S(\cdot, T_{max}-\delta)\|_{L^2}^2+
    \frac{1}{4 \epsilon}
    \int_{T_{max}-\delta}^{T_{max}} 
    \|f(\cdot,t)\|_{L^2}^2 dt
    <+\infty,
\end{equation}
which is a contradiction because $T_{max}<T^*$ implies that
\begin{equation}
    \limsup_{t \to T_{max}} \|S(\cdot,t)\|_{L^2}^2=+\infty.
\end{equation}
This completes the proof.
\end{proof}

Note that the boundary case in our paper is $q=\frac{3}{2},$ not $q=3.$ This is because the regularity criterion in \cite{SereginL3,Gallagher} is on $u,$ whereas our regularity criterion is on an eigenvalue of the strain matrix, which scales like $\nabla \otimes u.$ This is directly related to the Sobolev embedding $W^{1,\frac{3}{2}}(\mathbb{R}^3) \subset L^3(\mathbb{R}^3).$

Theorem \ref{RegCrit} is one of few regularity criteria for the Navier-Stokes equations involving a signed quantity, which is not too surprising, given that the Navier-Stokes equation is a vector valued equation. Even the scalar regularity criteria based on only one component of $u$ do not involve signed quantities \cite{Chemin}. The only other regularity criterion for the Navier-Stokes equation involving a signed quantity---at least to the knowledge of the author---is the regularity criterion proved by Seregin and \v{S}ver\'ak \cite{SereginPressure} that for a smooth solution to the Navier-Stokes equation to blowup in finite time, $p$ must become unbounded below and $p+\frac{1}{2}|u|^2$ must become unbounded above. 

We will also make a remark about the relationship between this result and the regularity criterion on one component of the gradient tensor $\frac{\partial u_j}{\partial x_i}$ in \cite{CaoTiti}. A natural question to ask in light of this regularity criterion is whether it is possible to prove a regularity criterion on just one entry of the strain tensor $S_{ij}.$ This paper does not answer this question, however we do prove a regularity criterion on just one diagonal entry of the diagonalization of the strain tensor.
\begin{corollary}[Any eigenvalue of strain characterizes the blow-up time] \label{RegCor}
Let $u \in C\left([0,T];\dot{H}^1\left (\mathbb{R}^3\right ) \right )
\cap L^2\left ([0,T];\dot{H}^2\left (\mathbb{R}^3\right )\right ),$ for all $T<T_{max}$ be a mild solution to the Navier-Stokes equation with force 
$f\in L^2_{loc}\left ((0,T^*);L^2\left (\mathbb{R}^3\right)\right )$.
If $\frac{2}{p}+\frac{3}{q}=2,$ with $\frac{3}{2}<q \leq + \infty,$ then
\begin{equation}
\|u(\cdot,T)\|_{\dot{H}^1}^2 \leq \left ( \left \|u^0\right \|_{\dot{H}^1}^2+\int_0^T\|f(\cdot,t)\|_{L^2}^2 dt \right )
\exp \left ({C_q \int_{0}^T\|\lambda_i(\cdot,t)\|_{L^q\left (\mathbb{R}^3\right )}^p dt} \right ),
\end{equation}
with the constant $C_q$ depending only on $q.$
In particular if $T_{max}<T^*,$ then 
\begin{equation}
\int_{0}^{T_{max}}\|\lambda_i(\cdot,t)\|_{L^q(\mathbb{R}^3)}^p dt= + \infty.
\end{equation}
\end{corollary}
\begin{proof}
$\lambda_1 \leq \lambda_2 \leq \lambda_3$ and $\lambda_1+\lambda_2+\lambda_3=0$ implies that $|\lambda_1|,|\lambda_3|\geq |\lambda_2| \geq |\lambda_2^+|.$ Therefore 
\begin{equation}
\int_0^T\|\lambda_2^+(\cdot,t)\|_{L^q}^p dt \leq 
\int_0^T\|\lambda_i(\cdot,t)\|_{L^q}^p dt.
\end{equation}
Applying this inequality to both conclusions in Theorem \ref{RegCrit}, this completes the proof.
\end{proof}
We will also note that there is a gap to be closed in the regularity criterion on $\frac{\partial u_j}{\partial x_i},$ because it is not the optimal result with respect to scaling and requires subcritical control on $\frac{\partial u_j}{\partial x_i}.$ That is, the result only holds for 
$\frac{2}{p}+\frac{3}{q}=\frac{q+3}{2q}<2,$ for $i\neq j$ and 
$\frac{2}{p}+\frac{3}{q}=\frac{3q+6}{4q}<2,$ for $i=j,$
whereas the regularity criterion on one of the eigenvalues in Corollary \ref{RegCor} is critical with respect to the scaling. It is natural, however, to ask whether the main theorem in this paper can be extended to the critical Besov spaces, so in that sense the result may be pushed further.

Corollary \ref{RegCor} is only really a new result, however, for $\lambda_2$. This is because $|\lambda_1|$ and $|\lambda_3|$ both control $|S|$. As we will see from the following proposition, the regularity criteria in terms of $\lambda_1$ or $\lambda_3$ follow immediately from the Prodi-Serrin-Ladyzhenskaya regularity criterion without needing to use strain evolution equation at all, so in this case Corollary 5.3 is just an unstated corollary of previous results.

\begin{proposition}[Lower bounds on the magnitude of the extermal eigenvalues]
Suppose $M\in S^{3 \times 3}$ is a symmetric trace free matrix with eigenvalues $\lambda_1 \leq \lambda_2 \leq \lambda_3.$ Then 
\begin{equation}
\lambda_3 \geq \frac{1}{\sqrt{6}}|S|,
\end{equation}
with equality if and only if $-\frac{1}{2}\lambda_1=\lambda_2=\lambda_3,$
and
\begin{equation}
\lambda_1 \leq -\frac{1}{\sqrt{6}}|S|,
\end{equation}
with equality if and only if $\lambda_1=\lambda_2=-\frac{1}{2}\lambda_3.$

Furthermore, for all $S \in L^2_{st}$ and for all $1\leq q \leq +\infty$
\begin{equation}
\|S\|_{L^q}\leq \sqrt{6}\|\lambda_1\|_{L^q}
\end{equation}
and
\begin{equation}
\|S\|_{L^q}\leq \sqrt{6}\|\lambda_3\|_{L^q}.
\end{equation} 
\end{proposition} \label{EigenStr}
\begin{proof}
We will prove the statement for $\lambda_3.$ The proof of the statement for $\lambda_1$ is entirely analogous and is left to the reader. 
First observe that if $-\frac{1}{2}\lambda_1=\lambda_2=\lambda_3,$ then
\begin{equation}
|S|^2=\lambda_1^2+\lambda_2^2+\lambda_3^2=6\lambda_3^2,
\end{equation}
So we have proven that if $\lambda_2=\lambda_3,$ 
then $\lambda_3=\frac{1}{\sqrt{6}}|S|.$
Now suppose $\lambda_2<\lambda_3.$
Recall that
\begin{equation}
\tr(M)=\lambda_1+\lambda_2+\lambda_3=0,
\end{equation}
so \begin{equation}
\lambda_1=-\lambda_2-\lambda_3.
\end{equation}
Therefore we find that
\begin{align}
|S|^2&=\lambda_1^2+\lambda_2^2+\lambda_3^2\\
&=(-\lambda_2-\lambda_3)^2+\lambda_2^2+\lambda_3^2\\
&=2\lambda_2^2+2\lambda_3^2+2\lambda_2 \lambda_3\\
&\leq 3\lambda_2^2+3\lambda_3\\
&< 6\lambda_3^2,
\end{align}
where we have applied Young's inequality and used the fact that $\lambda_2<\lambda_3.$
Noting that $\lambda_3 \geq 0$, this completes the proof. We leave the analogous proof for $\lambda_1$ to the reader. The $L^q$ bounds follow immediately from integrating these bounds pointwise when one recalls that $\tr(S)=0.$ We will note here that the $L^q$ norms may be infinite, as by hypothesis we only have $S\in L^2,$ but by convention the inequality is satisfied if both norms are infinite.
\end{proof}

In particular this implies that regularity criteria involving $\lambda_1$ or $\lambda_3$ follow immediately from regularity criteria involving $S,$ so while the regularity criteria on $\lambda_1$ and $\lambda_3$ in Corollary \ref{RegCor} do not appear in the literature to the knowledge of the author, these criteria do not offer a real advance over the Prodi-Serrin-Ladyzhenskaya criterion \cite{Prodi,Serrin,Ladyzhenskaya}, as the critical norm on $u$ can be controlled by the critical norm on $S$ using Sobolev embedding, which can in turn be bounded by the critical norm on $\lambda_1$ or $\lambda_3$ using Proposition \ref{EigenStr}. That is
\begin{equation}
\|u\|_{L^{q^*}}\leq C \|S\|_{L^q}\leq \sqrt{6}C\|\lambda_3\|_{L^q}.
\end{equation}
It is the regularity criterion in terms of $\lambda_2^+$ that is really significant, because it encodes geometric information about the strain beyond just its size.

We will also note that none of the regularity criteria involving $\nabla u_j$ \cite{ZhouOne}, $\partial_{x_i}u$ \cite{Kukavica}, or $\partial_{x_i}u_j$ \cite{CaoTiti}, have been proven for the Navier-Stokes equation with an external force. However, the regularity criterion in Theorem \ref{RegCrit} is also valid for Navier-Stokes equation with an external force. It may only be an exercise to extend the results cited above to the case with an external force, but because these papaers do not establish their regularity criteria by applying Gr\"onwall type estimates to the enstrophy, it is not immediately clear that this is is the case.

\begin{lemma}[The middle eigenvector is minimal] \label{Spectral}
Suppose $S \in L_{st}^2$ and $v \in L^\infty(\mathbb{R}^3;\mathbb{R}^3)$ with $|v(x)|=1$ almost everywhere $x \in \mathbb{R}^3.$ Then
\begin{equation}
|\lambda_2(x)|\leq \left |S(x)v(x)\right |
\end{equation}
almost everywhere $x\in \mathbb{R}^3.$
\end{lemma}
\begin{proof}
By the spectral theorem, we know that there is an orthonormal eigenbasis for $\mathbb{R}^n.$ In particular, take $v_1(x),v_2(x),v_3(x)$ to be eigenvectors of $S(x)$ corresponding to eigenvalues $\lambda_1(x),\lambda_2(x),\lambda_3(x)$ such that 
$|v_1(x)|,|v_2(x)|,|v_3(x)|=1$ almost everywhere $x\in \mathbb{R}^3.$ Then from the spectral theorem we know that $\left \{ v_1(x),v_2(x),v_3(x) \right \}$ is an orthonormal basis for $\mathbb{R}^3$ almost everywhere $x \in \mathbb{R}^3.$
Therefore
\begin{equation}
Sv=\lambda_1 (v \cdot v_1)v_1+ \lambda_2(v \cdot v_2)v_2+ 
\lambda_3 (v \cdot v_3) v_3.
\end{equation}
Therefore we can see that
\begin{equation}
|Sv|^2=\lambda_1^2 (v \cdot v_1)^2+ \lambda_2^2(v \cdot v_2)^2+ 
\lambda_3^2 (v \cdot v_3)^2.
\end{equation}
We know that $|\lambda_2|\leq |\lambda_1|,|\lambda_3|,$ so
\begin{equation}
|Sv|^2 \geq \lambda_2^2 \left ((v \cdot v_1)^2+ (v \cdot v_2)^2
+ (v \cdot v_3)^2 \right ).
\end{equation}
Because $\left \{ v_1(x),v_2(x),v_3(x) \right \}$ is an orthonormal basis for $\mathbb{R}^3$ almost everywhere $x \in \mathbb{R}^3,$ we conclude that
\begin{equation}
(v \cdot v_1)^2+ (v \cdot v_2)^2+ (v \cdot v_3)^2=|v|^2=1.
\end{equation}
Therefore
\begin{equation}
|Sv|^2 \geq \lambda_2^2.
\end{equation}
This concludes the proof.
\end{proof}

Now that we have proven Lemma \ref{Spectral}, we will prove a new regularity criterion for the strain tensor. This regularity criterion is Theorem \ref{IntroStrain} in the introduction, and is restated here for the reader's convenience.

\begin{theorem}[Blowup requires the strain to blow up in every direction] \label{StrainAllDirections}
Let $u \in C\left([0,T];\dot{H}^1\left (\mathbb{R}^3\right ) \right )
\cap L^2\left ([0,T];\dot{H}^2\left (\mathbb{R}^3\right )\right ),$ for all $T<T_{max}$ be a mild solution to the Navier-Stokes equation with force 
$f\in L^2_{loc}\left ((0,T^*);L^2\left (\mathbb{R}^3\right)\right )$, and let 
$v \in L^\infty\left (\mathbb{R}^3 \times [0,T_{max}];\mathbb{R}^3\right ),$ 
with $|v(x,t)|=1$ almost everywhere.
If $\frac{2}{p}+\frac{3}{q}=2,$ with $\frac{3}{2}<q \leq + \infty,$ then
\begin{equation}
\|u(\cdot,T)\|_{\dot{H}^1}^2 \leq 
\left ( \left\|u^0\right\|_{\dot{H}^1}^2+\int_0^T \|f(\cdot,t)\|_{L^2}^2 dt \right )
\exp \left (C_q \int_{0}^T\|S(\cdot,t)v(\cdot,t)\|_{L^q\left (\mathbb{R}^3\right )}^p dt \right ),
\end{equation}
with the constant $C_q$ depending only on $q.$
In particular if the maximal existence time for a mild solution $T_{max}<T^*,$ then 
\begin{equation}
\int_{0}^{T_{max}}\|S(\cdot,t)v(\cdot,t)\|_{L^q(\mathbb{R}^3)}^p dt= + \infty.
\end{equation}
\end{theorem}
\begin{proof}
This follows immediately from Lemma \ref{Spectral} and Theorem \ref{RegCrit}.
\end{proof}

We can use Theorem \ref{StrainAllDirections} to prove a new one-direction-type regularity criterion involving the sum of the derivative of the whole velocity in one direction, and the gradient of the component in the same direction. In fact, Theorem \ref{StrainAllDirections} allows us to prove a one direction regularity criterion that involves different directions in different regions of $\mathbb{R}^3.$ First off, for any unit vector 
$v\in \mathbb{R}^3, |v|=1$ we define $\partial_v=v\cdot \nabla$ and 
$u_v=u \cdot v.$
We will now prove Theorem \ref{OneDir}, which is restated here for the reader's convenience.

\begin{corollary}[Local one direction regularity criterion] \label{BlowupOneDirection}
Let $\left \{v_n(t) \right \}_{n\in \mathbb{N}}\subset \mathbb{R}^3$ with 
$|v_n(t)|=1.$ Let $\left \{\Omega_n(t) \right \}_{n\in \mathbb{N}}\subset \mathbb{R}^3$ be Lebesgue measurable sets such that for all $m \neq n,$
$\Omega_m(t) \cap \Omega_n(t)= \emptyset,$ and 
$\mathbb{R}^3= \bigcup_{n\in \mathbb{N}} \Omega_n(t).$
Let $u \in C\left([0,T];\dot{H}^1\left (\mathbb{R}^3\right ) \right )
\cap L^2\left ([0,T];\dot{H}^2\left (\mathbb{R}^3\right )\right ),$ for all $T<T_{max}$ be a mild solution to the Navier-Stokes equation with force 
$f\in L^2_{loc}\left ((0,T^*);L^2\left (\mathbb{R}^3\right)\right ).$
If $\frac{2}{p}+\frac{3}{q}=2,$ with $\frac{3}{2}<q \leq + \infty,$ then
\begin{equation}
\|u(\cdot,T)\|_{\dot{H}^1}^2 \leq 
\left ( \left\|u^0\right\|_{\dot{H}^1}^2 +\int_0^T \|f(\cdot,t)\|_{L^2}^2 dt \right)
\exp \left (C_q \int_{0}^T \left (\sum_{n=1}^\infty \bigl \|
\partial_{v_n}u(\cdot,t)+ \nabla u_{v_n}(\cdot,t)
\bigr \|_{L^q(\Omega_n(t))}^q \right )^\frac{p}{q} dt \right ),
\end{equation}
with the constant $C_q$ depending only on $q.$
In particular, if the maximal existence time for a mild solution $T_{max}<T^*,$ then 
\begin{equation} \label{alldirections}
\int_{0}^{T_{max}}\left (\sum_{n=1}^\infty \bigl\|
\partial_{v_n} u(\cdot,t)+ \nabla u_{v_n}(\cdot,t)
\bigr\|_{L^q(\Omega_n(t))}^q \right )^\frac{p}{q} dt= + \infty.
\end{equation}
In particular if we take $v_n(t)=\left (\begin{matrix} 0 \\ 0 \\ 1 \end{matrix} \right )$ for all $n\in \mathbb{N},$ then \eqref{alldirections} reduces to
\begin{equation}
    \int_0^{T_{max}}\|\partial_3u(\cdot,t)+\nabla u_3(\cdot,t)\|_{L^q}^p dt=+\infty.
\end{equation}
\end{corollary}
\begin{proof}
Let $v(x,t)=\sum_{n=1}^\infty v_n(t) I_{\Omega_n(t)}(x),$ where
$I_\Omega$ is the indicator function $I_\Omega(x)=1$ for all $x \in \Omega$ and 
$I_\Omega(x)=0$ otherwise.
Note that in this case we clearly have
\begin{equation}
S(x,t)v(x,t)=\sum_{n=1}^\infty I_{\Omega_n(t)}(x)S(x,t)v_n(t).
\end{equation}
Because $\left \{\Omega_n \right \}_{n \in \mathbb{N}}$ are disjoint, we have
\begin{equation}
\|S(\cdot,t)v(\cdot,t)\|_{L^q\left (\mathbb{R}^3 \right )}^q=
\sum_{n=1}^\infty \|S(\cdot,t)v_n(t)\|_{L^q(\Omega_n(t))}^q.
\end{equation}
Therefore we find that
\begin{equation}
\|S(\cdot,t)v(\cdot,t)\|_{L^q\left (\mathbb{R}^3 \right )}^p=
\left ( \sum_{n=1}^\infty \|S(\cdot,t)
v_n(t)\|_{L^q(\Omega_n(t))}^q \right )^\frac{p}{q}.
\end{equation}
Finally observe that
\begin{equation}
S(x,t)v_n(t)=\frac{1}{2}\partial_{v_n}u(x,t)
+\frac{1}{2}\nabla u_{v_n}(x,t),
\end{equation}
so
\begin{equation}
\|S(\cdot,t)v(\cdot,t)\|_{L^q\left (\mathbb{R}^3 \right )}^p=
\left ( \sum_{n=1}^\infty \left \|
\frac{1}{2}\partial_{v_n}u(\cdot,t)+ \frac{1}{2}\nabla u_{v_n}(\cdot,t)
\right \|_{L^q(\Omega_n(t))}^q \right )^\frac{p}{q}.
\end{equation}
Applying Theorem \ref{StrainAllDirections}, this completes the proof.
\end{proof}

Corollary \ref{BlowupOneDirection} significantly extends the range of exponents for which we have a critical regularity criterion involving one direction. For instance, Kukavica and Ziane \cite{Kukavica} showed that if $T_{max}<+\infty,$ and if 
$\frac{2}{p}+\frac{3}{q}=2,$ with $\frac{9}{4} \leq q \leq 3,$ then
\begin{equation}
\int_{0}^{T_{max}}\|\partial_3 u(\cdot,t)\|_{L^q(\mathbb{R}^3)}^p dt= + \infty.
\end{equation}
More recently, it was shown by Chemin, Zhang, and Zhang \cite{Chemin2,Chemin3} that if $T_{max}<+\infty$ and \newline $4<p<+\infty,$ then
\begin{equation}
\int_0^{T_{max}}\|u_3(\cdot,t)\|_{\dot{H}^{\frac{1}{2}+\frac{2}{p}}}^p=+\infty.
\end{equation}
In the case where there is no external force, $f=0,$ these results imply the special case of Corollary \ref{BlowupOneDirection}, that if $T_{max}<+\infty$ then \begin{equation}\int_0^{T_{max}}\|\partial_3 u(\cdot,t)+\nabla u_3(\cdot,t)\|_{L^q}^p=+\infty,
\end{equation}
in the range of exponents $\frac{9}{4}\leq q \leq 3$ and $\frac{3}{2}<q<6$ respectively. This follows from the Helmholtz decomposition, as we will now show.

\begin{proposition}[Helmholtz decomposition] \label{Hodge}
Suppose $1<q<+\infty.$ For all $v\in 
L^q(\mathbb{R}^3;\mathbb{R}^3)$ there exists a unique $u\in L^q(\mathbb{R}^3;\mathbb{R}^3),$ $\nabla \cdot u=0$ and $\nabla f \in L^q(\mathbb{R}^3;\mathbb{R}^3)$ such that $v=u+\nabla f.$ Note because we do not have any assumptions of higher regularity, we will say that $\nabla \cdot u=0,$ if for all $\phi \in C_c^\infty(\mathbb{R}^3)$
\begin{equation}
    \int_{\mathbb{R}^3}u \cdot \nabla \phi=0,
\end{equation}
and we will say that $\nabla f$ is a gradient if for all $w\in C_c^\infty(\mathbb{R}^3;\mathbb{R}^3), \nabla \cdot w=0,$ we have
\begin{equation}
    \int_{\mathbb{R}^3}\nabla f \cdot w=0.
\end{equation}
Furthermore there exists $B_q\geq 1$ depending only on $q,$ such that 
\begin{equation}
    \|u\|_{L^q}\leq B_q \|v\|_{L^q},
\end{equation}
and
\begin{equation}
    \|\nabla f\|_{L^q} \leq B_q \|v\|_{L^q}.
\end{equation}
\end{proposition}
\begin{proof}
This is a well-known, classical result. For details, see for instance \cite{NS21}. We will also note here that the $L^q$ bounds here are equivalent to the $L^q$ boundedness of the Riesz transform. Take the Riesz transform to be given by $R=\nabla (-\Delta)^{-\frac{1}{2}},$ then 
$P_{df}(v)=R\times (R \times v),$ and
$P_{g}(v)=-R (R \cdot v).$
\end{proof}
Based on Proposition \ref{Hodge}, we will now define the projections onto the space of gradients and the space of divergence free vector fields.
\begin{definition}
    Fix $1<q<+\infty.$ Define $P_{df}:L^q(\mathbb{R}^3;\mathbb{R}^3) \to
    L^q(\mathbb{R}^3;\mathbb{R}^3)$ and
    $P_{g}:L^q(\mathbb{R}^3;\mathbb{R}^3) \to
    L^q(\mathbb{R}^3;\mathbb{R}^3)$ by 
    $P_{df}(v)=u$ and $P_{g}(v)=\nabla f,$ where $v, u,$ and $\nabla f$ are taken as above in Proposition \ref{Hodge}.
\end{definition}
Now we observe that these projections allow us to control $\|\partial_3 u\|_{L^q}$ and $\|u_3\|_{\dot{H}^{\frac{1}{2}+\frac{p}{2}}}$ by 
$\|\partial_3 u+\nabla u_3\|_{L^q}.$
In particular, we find
\begin{equation}
\|\partial_3 u\|_{L^q}=\|P_{df}\left (
\partial_3 u+\nabla u_3\right)\|_{L^q}
\leq B_q\|\partial_3 u +\nabla u_3\|_{L^q}.
\end{equation}
Applying the the Sobolev embedding $\dot{H}^{\frac{1}{2}+\frac{p}{2}}\left (\mathbb{R}^3 \right ) \subset W^{1,q}\left (\mathbb{R}^3 \right )$
when $\frac{2}{p}+\frac{3}{q}=2,$ and the $L^q$ boundedness of $P_g,$ we can also see that
\begin{equation}
\|u_3\|_{\dot{H}^{\frac{1}{2}+\frac{2}{p}}}\leq D\|\nabla u_3\|_{L^q} \leq
D B_q\|\partial_3 u+\nabla u_3\|_{L^q}.
\end{equation}
This means that the regularity criterion requiring $\partial_3 u+\nabla u_3 \in L_t^p L_x^q$ is not new in the range $\frac{3}{2} < q \leq 6.$ It does, however, extend the range of exponents for which a scale critical, one direction type regularity criterion is known to $6\leq q \leq +\infty.$

More importantly, Corollary \ref{StrainAllDirections} and Corollary \ref{BlowupOneDirection} do not require regularity in a fixed direction, but allow this direction to vary. One interpretation of component reduction results for Navier-Stokes regularity criteria, is that if the solution is approximately two dimensional, then it must be smooth. The only reason that we have component reduction regularity criteria for the 3D Navier-Stokes equation, is because the 2D Navier-Stokes equation has smooth solutions globally in time. All of the previous component reduction regularity criteria involve some fixed direction, and so can be interpreted as saying if a solution is globally approximately two dimensional, then it must be smooth. Corollary \ref{StrainAllDirections} and Corollary \ref{BlowupOneDirection} strengthen these statements to the requirement that the solution must be regular even if it is only locally two dimensional. This shows the deep geometric significance of the Theorem \ref{RegCrit}, that $\lambda_2^+$ controls the growth of enstrophy.

\section{Blowup for a toy model ODE}

Now that we have outlined the main advances in terms of regularity criteria that are made possible by utilizing strain equation, we will consider a toy model ODE.
The main advantage of the strain equation formulation of the Navier-Stokes equation compared with the vorticity formulation is that the quadratic term $S^2+\frac{1}{4}w \otimes w$ has a much nicer structure then the quadratic term $S \omega$ in 
the vorticity formulation.  The price we pay for this is that there are additional terms, particularly $\Hess(p),$ which are not present in the vorticity formulation.  There is also the related difficulty that the consistency condition in the strain formulation is significantly more complicated than in the vorticity formulation.

We will now examine a toy model ODE, prove the existence and stability of blowup, and examine asymptotic behavior near blowup. The simplest toy model equation would be to keep only the local part of the quadratic term (vorticity depends non-locally on $S$), and to study the ODE
$\partial_t M+M^2=0$.  As long as the initial condition $M(0)$ is an invertible matrix, this has the solution $\left (M(t) \right )^{-1}= \left(M(0) \right )^{-1}+ tI_3$.  This equation will blow up in finite time assuming that $M(0)$ has at least one negative eigenvalue.  Blowup is unstable in general, because any small perturbation into the complex plane will mean there will not be blowup.  However, if we restrict to symmetric matrices, then blowup is stable, because then the eigenvalues must be real valued, so a small perturbation will remain on the negative real axis.  The negative real axis is an open set of $\mathbb{R}$, but not of $\mathbb{C}$, so blowup is stable only when we are restricted to matrices with real eigenvalues, which is the case we are concerned with as the strain tensor is symmetric.  This equation does not preserve the family of trace free matrices however, because $tr(M^2)=|M|^2\neq0$, and therefore doesn't really capture any of the features of the strain equation \eqref{NavierStrain}.  We will instead take our toy model ODE on the space of symmetric, trace free matrices to be
\begin{equation} \label{ToyModel}
\partial_t M+M^2-\frac{1}{3}|M|^2 I_3=0.
\end{equation}

Because every symmetric matrix is diagonalizable over $\mathbb{R}$, and every diagonalizable matrix is mutually diagonalizable with the identity matrix, this equation can be treated as a system of ODEs for the evolution of the eigenvalues 
$\lambda_1 \leq \lambda_2 \leq \lambda_3,$ with for every $1 \leq i \leq 3,$
\begin{equation} \label{EigenvalueODE}
\partial_t \lambda_i = -\lambda_i^2 +\frac{1}{3}(\lambda_1^2 +\lambda_2^2+ \lambda_3^3).
\end{equation}
This equation has two families of solutions with a type of scaling invariance.  Let $S(0)=C diag(-2,1,1)$, with $C>0$ then $S(t)= f(t) diag(-2,1,1)$, where $f_t=f^2, f(0)=C$.  Therefore we have blowup in finite time, with $S(t)=\frac{1}{\frac{1}{C}-t}diag(-2,1,1).$ 
The reverse case, one positive eigenvalue and two equal negative eigenvalues, also preserves scaling, but decays to zero as $t \to \infty.$  Let $S(0)=C diag(-1,-1,2)$, with $C>0$.  Then $S(t)=\frac{1}{\frac{1}{C}+t}diag(-1,-1,2)$.

We will show that the blow up solution is stable, while the decay solution is unstable.  Furthermore the blow up solution is asymptotically a global attractor except for the unstable family of solutions that decay to zero (i.e two equal negative eigenvalues and the zero solution).
To prove this we will begin by rewriting our system.  First of all, we will assume without loss of generality, that $S \neq 0$, because clearly if $S(0)=0$, then $S(t)=0$ is the solution.  If $S \neq 0$, then clearly $\lambda_1<0$ and $\lambda_3>0$.  Our system of equations really only has two degrees of freedom, because of the condition $tr(S)= \lambda_1+ \lambda_2+ \lambda_3=0$, but because we are interested in the ratios of the eigenvalues asymptotically, we will reduce the system to the two parameters $\lambda_3$ and
$r=-\frac{\lambda_1}{\lambda_3}$. These two parameters completely determine our system because $\lambda_1=-r \lambda_3$ and $\lambda_2=-\lambda_1-\lambda_3= (r-1)\lambda_3$.  We now will rewrite our system of ODEs:
\begin{align}
\partial_t \lambda_3&= \frac{1}{3}(\lambda_1^2+\lambda_2^2-2\lambda_3^2)\\
&=\frac{1}{3}\lambda_3^2\left (r^2+(r-1)^2 +2\right)\\
&=\frac{1}{3}\lambda_3^2\left (2r^2 -2r-1  \right ),
\end{align}
and
\begin{align}
\partial_t r&= \frac{\lambda_1 \partial_t \lambda_3 -\lambda_3 \partial_t \lambda_1}{\lambda_3^2}\\
&= \lambda_3 \left( -r(-\frac{1}{3}- \frac{2}{3}r+\frac{2}{3}r^2) +(-\frac{2}{3}+\frac{2}{3}r+\frac{1}{3}r^2) \right)\\
&=\frac{1}{3} \lambda_3 (-2r^3+3r^2+3r-2).
\end{align}

At this point it will be useful to remark on the range of values our two variables can take.  Clearly the largest eigenvalue $\lambda_3 \geq 0$, and $\lambda_3=0$ if and only if $\lambda_1,\lambda_2,\lambda_3=0$.  Now we turn to the range of values for $r$.  Recall that $\lambda_2=(r-1)\lambda_3$, and that $\lambda_1 \leq \lambda_2 \leq \lambda_3$.  Therefore $-r \leq r-1 \leq 1$, so $\frac{1}{2} \leq r \leq 2$.  If we take $f(r)=-2r^3+3r^2+3r-2$, we find that $f(r)$ is positive for $\frac{1}{2} < r < 2,$ with 
$f(\frac{1}{2}),f(2)=0$. This is the basis for the blowup solution being the asymptotic attractor. We are now ready to state our theorem on the existence and asymptotic behaviour of finite time blow up solutions.

\begin{theorem}[Toy model dynamics]
Suppose $\lambda_3(0)>0$ and $ r(0)>\frac{1}{2}$, then there exists $T>0$ such that $\lim_{t \to T}\lambda_3(t)=+ \infty$, and furthermore $\lim_{t \to T}r(t)=2$
\end{theorem}
\begin{proof}
We'll start by showing that finite time blow up exists, and then we will show that $r$ goes to $2$ as we approach the blow up time.
First we observe that $g(r)=2r^2-2r-1$, has a zero at $\frac{1+\sqrt{3}}{2}$. $g(r)<0$, for $\frac{1}{2} \leq r <\frac{1+\sqrt{3}}{2} $, and $g$ is both positive and increasing on $\frac{1+\sqrt{3}}{2} < r \leq 2$.
We will begin with the case where $r(0)=r_0>\frac{1+\sqrt{3}}{2}$.  Clearly $\partial_t r\geq 0$, so $r(t)>r_0$, and $g\left ( r(t)\right )>g(r_0).$  Let $C= \frac{1}{3}g(r_0)$, then we find that:
\begin{equation}
\partial_t \lambda_3 =\frac{1}{3}g\left ( r(t)\right ) \lambda_3^2 \geq C \lambda_3^2.
\end{equation}
From this differential inequality, we find that
\begin{equation}
\lambda_3(t) \geq \frac{1}{\frac{1}{\lambda_3(0)}-Ct},
\end{equation}
so clearly there exists a time $T \leq\frac{1}{C \lambda_3(0)}$, such that $\lim_{t \to T}\lambda_3(t)=+ \infty.$

Now we consider the case where $\frac{1}{2}<r_0 \leq \frac{1+\sqrt{3}}{2}$.  It suffices to show that there exists a $T_a>0$ such that $r(T_a)>\frac{1+\sqrt{3}}{2}$, then the proof above applies.
Note that $g$ is increasing on the interval $\left[-\frac{1}{2}, 2 \right]$, so $g\left ( r(t) \right )> g(r_0)$. Let $B=-\frac{1}{3}g(r_0)>0$, and let $C=\frac{1}{3} \min \left (f(r_0),f(\frac{1+\sqrt{3}}{2}) \right )$.  Suppose towards contradiction that for all $t>0$,
$r(t) \leq \frac{1+\sqrt{3}}{2}$.  Then we will have the differential inequalities,
\begin{equation} \label{inequality1}
\partial_t r \geq C \lambda_3,
\end{equation}
\begin{equation} \label{inequality2}
\partial_t \lambda_3 \geq -B \lambda_3^2.
\end{equation}
From \eqref{inequality2} it follows that
\begin{equation} \label{EigenEstimate}
\lambda_3(t) \geq \frac{1}{\frac{1}{\lambda_3(0)}+Bt}.
\end{equation}
Plugging \eqref{EigenEstimate} into \eqref{inequality1}, we find that
\begin{equation} \label{contradict}
r(t) \geq r_0 +C\int_0^t \  \frac{1}{\frac{1}{\lambda_3(0)}+B \tau} d\tau=r_0+ \frac{C}{B} \log \left(1+ B \lambda_3(0)t \right).
\end{equation}
However, this estimate \eqref{contradict} clearly contradicts our hypothesis that $r(t) \leq \frac{1+\sqrt{3}}{2}$ for all $t>0$.
Therefore, we can conclude that there exists $T_a>0$, such that $r(T_a) > \frac{1+\sqrt{3}}{2}$, and then we have reduced the problem to the case that we have already proven.

Now we will show that $\lim_{t \to T}r(t)=2$.  Suppose toward contradiction that $\lim_{t \to T}r(t)=r_1<2$.  First take $a(t)=\frac{1}{3} f\left (r(t) \right )$.  Observe that $a(t)>0$ for $0 \leq t \leq T$.  
Our differential equation is now given by $\partial_t \lambda_3 =a(t) \lambda_3^2$, which must satisfy
\begin{equation} \label{ODE1}
\frac{1}{\lambda_3(t_1)}-\frac{1}{\lambda_3(t_2)}=\int_{t_1}^{t_2} a(\tau)d\tau.
\end{equation}
If we take $t_2=T$, the blow up time, then \eqref{ODE1} reduces to 
\begin{equation} \label{ODE2}
\frac{1}{\lambda_3(t)}=\int_{t}^{T} a(\tau)d\tau .
\end{equation}
Let $A(t)=\int_{t}^T a(\tau)d\tau$.  Clearly $A(T)=0, A'(T)=-a(T)<0$.  By the fundamental theorem of calculus, for all $m>a(T)$, there exists $\delta>0$, such that for all $t, T-\delta<t<T$,
\begin{equation} \label{FunCalc}
A(t) \leq -m(t-T)=m(T-t).
\end{equation}
Using the definition of $A$ and plugging in to \eqref{ODE2} we find that for all $T=\delta<t<T$,
\begin{equation} \label{EigenGrowth}
\lambda_3(t) \geq \frac{1}{m(T-t)}.
\end{equation}
Let $B=\frac{1}{3} min \left ( f(r_0), f(r_1) \right )$.  It then follows from our hypothesis that 
\begin{equation} \label{EigenODE}
\partial_t r\geq B \lambda_3.
\end{equation}
Therefore we can apply the estimate \eqref{EigenGrowth} to the differential inequality \eqref{EigenODE} to find that for all $T-\delta<t<T$,
\begin{equation} \label{rGrowth}
r(t)\geq r(T-\delta)+B \int_{T-\delta}^t \frac{1}{m(T-\tau)} d\tau= r(T-\delta) + \frac{B}{M} \log \left ( \frac{\delta}{T-t} \right ).
\end{equation}
However, it is clear from \eqref{rGrowth} that $\lim_{t \to T}r(t)= + \infty$, contradicting our hypothesis that $\lim_{t \to T}r(t)<2$,
so we can conclude that $\lim_{t \to T}r(t)=2$.
\end{proof}

This toy model ODE shows that the local part of the quadratic nonlinearity tends to drive the intermediate eigenvalue $\lambda_2$ upward to $\lambda_3$, unless $\lambda_1=\lambda_2.$ Given the nature of the regularity criterion on $\lambda_2^+,$ the dynamics of the eigenvalues of the strain matrix are extremely important. The fact that the toy model ODE blows up from all initial conditions where $\lambda_1<\lambda_2,$ and that $\lambda_2=\lambda_3$ is a global attractor on all initial conditions where $\lambda_1<\lambda_2,$ provides a mechanism for blowup, but of course the very complicated nonlocal effects make it impossible to say anything definitive about blowup for the full Navier-Stokes strain equation without a much more detailed analysis.

\pagebreak
\bibliographystyle{plain}
\bibliography{Bib}

\end{document}